\documentclass[11pt,reqno]{amsart}
\usepackage[T1]{fontenc}
\usepackage[utf8]{inputenc}
\usepackage{lmodern}
\usepackage{amsmath,amssymb,mathtools}
\usepackage{esint}
\usepackage[letterpaper,hmargin=1.125in,vmargin=1in]{geometry}
\usepackage{microtype}
\usepackage{enumitem}
\usepackage{xcolor}
\usepackage{longtable}
\usepackage[colorlinks=true,linkcolor=blue!45!black,citecolor=blue!45!black,urlcolor=blue!45!black]{hyperref}
\hypersetup{pdftitle={Measure contraction property on isometric leaves and monotone fibres},pdfauthor={Krzysztof J. Ciosmak}}

\newtheorem{theorem}{Theorem}[section]
\newtheorem{lemma}[theorem]{Lemma}
\newtheorem{proposition}[theorem]{Proposition}
\newtheorem{corollary}[theorem]{Corollary}
\theoremstyle{remark}
\newtheorem{remark}[theorem]{Remark}
\theoremstyle{plain}
\newtheorem{maintheorem}{Theorem}

\numberwithin{equation}{section}

\newcommand{\R}{\mathbb R}
\newcommand{\N}{\mathbb N}
\renewcommand{\S}{\mathcal S}
\newcommand{\T}{T}
\newcommand{\Hh}{\mathcal H}
\newcommand{\CC}{CC}
\newcommand{\norm}[1]{\lVert#1\rVert}
\newcommand{\ip}[2]{\langle#1,#2\rangle}
\newcommand{\ri}{\operatorname{int}}
\DeclareMathOperator{\cl}{cl}
\DeclareMathOperator{\Aff}{Aff}
\DeclareMathOperator{\Conv}{Conv}
\DeclareMathOperator{\dist}{dist}
\DeclareMathOperator{\supp}{supp}
\DeclareMathOperator{\dom}{dom}
\DeclareMathOperator{\ran}{ran}
\DeclareMathOperator{\graph}{graph}
\DeclareMathOperator{\Id}{Id}
\DeclareMathOperator{\Lip}{Lip}
\DeclareMathOperator{\MCP}{MCP}
\DeclareMathOperator{\CD}{CD}
\DeclareMathOperator{\Ric}{Ric}

\DeclareMathOperator*{\argmin}{argmin}
\newcommand{\TV}{\mathrm{TV}}
\newcommand{\dd}{\,d}
\setlist[enumerate,1]{label=(\roman*),leftmargin=2.2em,itemsep=2pt,topsep=4pt}

\title[Measure contraction property on isometric leaves]{Measure contraction property on isometric leaves and monotone fibres}
\author{Krzysztof J. Ciosmak}
\address{Beijing Institute of Mathematical Sciences and Applications, No. 544, Hefangkou Village, Huaibei Town, Huairou District, Beijing 101408, China}
\email{kciosmak@bimsa.cn}
\date{8 September 2026}

\begin{document}
\begin{abstract}
For finite measures with positive densities on convex Euclidean
supports, we prove that $\MCP(\kappa,N)$ passes with unchanged
parameters to almost every isometric leaf of an arbitrary
nonexpansive map. The proof rests on a sharp contraction inequality
for geometric conditional densities, with exponent equal to the
leaf codimension. The inherited dimension parameter is optimal.
A total-variation limit on resolvent graphs extends the result
to inverse fibres of maximal monotone relations, including convex gradient fibres.
We also disprove Klartag's curvature-dimension inheritance
conjecture by a firmly nonexpansive example in dimension three
and a gradient example in dimension four.
In codimension one, affinity of the geometric density yields
curvature-dimension inheritance.
The first example also gives failure on monotone fibres.
Both constructions admit arbitrarily large curvature loss,
including for a fixed Gaussian ambient measure on families
of leaves of positive quotient measure.
\end{abstract}
\maketitle

\section{Introduction}\label{sec:introduction}

Klartag conjectured that synthetic lower Ricci curvature bounds and upper
dimension bounds are inherited by the conditional measures on isometric
leaves of vector-valued nonexpansive maps \cite[Chapter~6]{r24}. We show
that this is false, even for firmly nonexpansive maps in dimension three,
and establish a general replacement: the measure contraction property is
inherited by almost every leaf, on every dimensional stratum, with
unchanged curvature and dimension parameters. The result also extends to
inverse fibres of maximal monotone relations and, in particular, to fibres
of the gradient of an arbitrary convex function. In codimension one we
prove the stronger curvature-dimension inheritance, thereby identifying
dimension three as the smallest ambient dimension in which Klartag's conjecture
can fail.

A leaf of a $1$-Lipschitz map $u:\R^n\to\R^m$ is a maximal set on which
$u$ preserves distances. Leaves are closed and convex, and almost every
point belongs to a unique leaf \cite{r16}. Consequently, every finite
measure $\mu\ll\lambda$ admits a normalised disintegration
\begin{equation}\label{eq:1.1}
\mu=\int_{\CC(\R^n)}\mu_{\S}\dd\nu(\S).
\end{equation}
Here $\lambda$ denotes Lebesgue measure, $\CC(\R^n)$ is the space of
nonempty closed convex subsets of $\R^n$, and the conditional measures
$\mu_{\S}$ have mass one. All leaf metrics are restrictions of the
 Euclidean metric.

We use the curvature-dimension condition $\CD(\kappa,N)$
of Lott--Villani and Sturm \cite{r26,r27,r32}.
In our weighted Euclidean setting, it can be expressed by
a lower bound on the Bakry--\'Emery tensor.
Section~\ref{sec:curvature} recalls the precise characterisation
for positive $\mathcal C^2$ densities on closed convex supports.
In particular, $\CD(0,\infty)$ is equivalent in this setting
to log-concavity of the density.

The convexity of the leaves does not, by itself, settle the inheritance
problem. Although the intrinsic metric is Euclidean, the conditional
measure retains information about the arrangement of the leaves in the
ambient space. If $\mu$ has density $h$, then its conditional density on
a $k$-dimensional leaf $\S$ is not generally just a normalised restriction of
$h$: an additional geometric factor $g_{\S}$ appears. The distortion of
this factor determines whether the ambient measure-contraction or
curvature-dimension bound survives disintegration.

A simple example is the radial decomposition induced by
$u(x)=\norm{x}$, $x\in\mathbb{R}^n$. Normalised Lebesgue measure on the unit ball in $\R^n$
induces, on each radial segment, the conditional measure
\begin{equation*}
    d\mu_\theta(r)=n r^{n-1}\,dr\text{ with } 0<r<1.
\end{equation*}
Contraction towards the origin therefore multiplies the conditional mass
of a Borel set by $t^n$. The conditional measure is supported
on a one-dimensional set, but its contraction behaviour still records the
ambient dimension. The geometric density in this example is proportional
to $r^{n-1}$; the general estimate that we will establish replaces this exponent by the
codimension of the leaf.


\subsection{Measure contraction property and the geometric density}

We use the measure contraction property $\MCP(\kappa,N)$ in Ohta's
normalisation \cite[Definition~2.1]{r29}; see also \cite{r32}.
For $\kappa\in\R$ and $1<N<\infty$, a finite measure $\eta$ with
closed convex support $Y$ in a Euclidean affine space satisfies this
condition precisely when
\begin{equation*}
    \eta\bigl(H_{o,t}(E)\bigr)
 \ge \int_E\beta_{\kappa,N}^{(t)}(\norm{x-o})\dd\eta(x)\text{, where } H_{o,t}(x)=o+t(x-o),
\end{equation*}
for every $o\in Y$, $0<t\le1$ and Borel set
$E\subset Y\cap B(o,R_{\kappa,N})$; compare
\cite[Lemma~2.3]{r29}. Here $\beta_{\kappa,N}^{(t)}$ is the distortion
coefficient defined in \eqref{eq:2.4}, and
$R_{\kappa,N}=\pi\sqrt{(N-1)/\kappa}$ for $\kappa>0$, while
$R_{\kappa,N}=\infty$ for $\kappa\le0$. In particular,
$\beta_{0,N}^{(t)}=t^N$, so the zero-curvature condition reads
\[
 \eta\bigl(o+t(E-o)\bigr)\ge t^N\eta(E).
\]
Thus $\MCP$ controls contraction towards a point, rather than the
interpolation between two probability distributions tested by $\CD$.
Our first theorem shows that disintegration along isometric leaves
preserves this contraction bound with exactly the same parameters.

\begin{maintheorem}[Measure contraction property on leaves]\label{thm:A}
Let $X\subset\R^n$ be closed and convex with nonempty interior, and let
$d\mu=h\,d\lambda|_X$ be finite, with $0<h<\infty$ almost everywhere.
Suppose that $(X,\norm{\cdot},\mu)$ satisfies $\MCP(\kappa,N)$, where
$\kappa\in\R$, $1<N<\infty$ and $N\ge n$. For every $1$-Lipschitz map
$u:\R^n\to\R^m$, almost every leaf satisfies
\begin{equation*}
    \supp\mu_{\S}=\S\cap X,\quad
\mu_{\S}\sim\Hh^{\dim\S}|_{\S\cap X},
\end{equation*}
and $(\S\cap X,\norm{\cdot},\mu_{\S})$ satisfies $\MCP(\kappa,N)$.
The dimension parameter $N$ cannot be reduced in general.
\end{maintheorem}

Theorem~\ref{thm:4.1} proves the assertion outside a single exceptional
family of leaves, independently of the centre and the contraction time.
Proposition~\ref{prop:4.3} proves sharpness even for scalar maps and ambient
$\CD(0,N)$ measures with positive smooth density in the interior of a ball.
Thus the retention of the dimensional parameter $N$, rather than a smaller parameter suggested by
the leaf dimension, is necessary.

The underlying geometric assertion requires no curvature assumption.
Theorem~\ref{thm:3.1} constructs geometric densities $g_{\S}$, determined
up to a positive, multiplicative constant on each leaf and depending only on the
decomposition. For every finite absolutely continuous measure
$d\mu=h\,d\lambda$, with $h\ge0$, its conditional measures satisfy
\begin{equation}\label{eq:1.2}
 d\mu_{\S}
 =\frac{h g_{\S}}{\int_{\S}h g_{\S}\dd\Hh^k}
   \dd\Hh^k|_\S
\end{equation}
on $\nu$-almost every $k$-dimensional leaf, with a positive finite
denominator. For a measure supported on $X$, the ambient density is
extended by zero outside $X$. The relative boundaries carry no
conditional mass. For $k>0$, the geometric density is positive on
$\ri\S$, its logarithm is locally Lipschitz there, and
\begin{equation}\label{eq:1.3}
 g_{\S}(o+t(x-o))\ge t^{n-k}g_{\S}(x)
\text{ for all }o\in\S,\ x\in\ri\S,\ 0<t\le1.
\end{equation}
The construction applies simultaneously on all dimensional strata.
Corollary~\ref{cor:3.8} also gives sharp directional estimates for
$\log g_{\S}$ in terms of distances to the ends of leaf chords. In
particular, $g_{\S}$ is constant along every complete line contained in
$\ri\S$.

The codimension exponent explains why the ambient dimension parameter
is preserved. In the zero-curvature case, the ambient density inequality
supplies a factor $t^{N-n}$, while \eqref{eq:1.3} supplies $t^{n-k}$.
The Jacobian of contraction within a $k$-dimensional leaf is $t^k$, so the product of these factors is $t^N$.
For nonzero curvature, the same cancellation preserves the full ambient
distortion coefficient. Combining the geometric estimate with the
pointwise ambient inequality of Lemma~\ref{lem:2.1} therefore yields
Theorem~\ref{thm:A}.

The main difficulty is to obtain \eqref{eq:1.3} without assuming a smooth
transverse parametrisation of the leaves. We compare the volumes of
transverse slices instead of differentiating the field of affine hulls.
The starting point is the nonnegative equality defect
\begin{equation*}
     c_u(x,y)=\norm{x-y}^2-\norm{u(x)-u(y)}^2\text{ for } x,y\in\mathbb{R}^n.
\end{equation*}
At a relative interior point of a leaf, vanishing of this defect
identifies points on the same leaf. Finite-target minimisation of the
defect gives disjoint interpolating segments in the transverse
comparison. Approximating the correspondence between slices by these
maps yields a volume estimate. This estimate first shows that the union
of relative leaf boundaries is Lebesgue-negligible, then gives equivalence of
the slice measures and constructs the conditional densities. A final
regularity argument produces the positive continuous representatives
for which \eqref{eq:1.3} holds at every admissible pair of points.

\subsection{Monotone fibres}

The second family of decompositions comes from monotone operators.
A relation $A:\R^n\rightrightarrows\R^n$ is called monotone if
\begin{equation*}
    \langle p-q,x-y\rangle\ge0
\text{ for all } p\in A(x), q\in A(y),\, x,y\in\mathbb{R}^n
\end{equation*}
For a single-valued map $T:D\subset\R^n\to\R^n$, this means
$\langle T(x)-T(y),x-y\rangle\ge0$ for all $x,y\in D$.
A monotone relation is maximal monotone if its graph is not properly
contained in any other monotone graph. We refer to Alberti--Ambrosio
\cite{r1} for the finite-dimensional geometric viewpoint and to
Bauschke--Combettes \cite{r4} for the general operator theory.

The fibres considered below are the inverse values
$A^{-1}(p)=\{x\in\mathbb{R}^n\mid p\in A(x)\}$, which are closed and convex when $A$ is
maximal monotone \cite{r4}. A basic example is $A=\partial f$, where
$f$ is proper, lower-semicontinuous and convex: its subdifferential is
maximal monotone \cite[Theorem~A]{r31}. At points $x$ where $f$ is
differentiable, $\partial f(x)=\{Df(x)\}$, and the inverse values
$(\partial f)^{-1}(p)$ are the contact sets with supporting affine
functions of slope $p$. The passage from isometric leaves to monotone
fibres is made through the resolvent.

\begin{maintheorem}[Measure contraction on monotone fibres]\label{thm:B}
Under the ambient assumptions of Theorem~\ref{thm:A}, let
$A:\R^n\rightrightarrows\R^n$ be maximal monotone and suppose that
$\mu(\R^n\setminus\dom A)=0$. Disintegrate $\mu$ with respect to any
Borel selection of $A$. For almost every label $p$, put
$F_p=A^{-1}(p)$ and $Y_p=F_p\cap X$. Then
\[
 \supp\mu_p=Y_p,\quad
 \mu_p\sim\Hh^{\dim F_p}|_{Y_p},
\]
and $(Y_p,\norm{\cdot},\mu_p)$ satisfies $\MCP(\kappa,N)$.
The disintegration is independent of the selection. In particular,
these conclusions hold for $A=\partial f$, whenever $f$ is proper,
lower-semicontinuous, convex and finite on $\ri X$.
\end{maintheorem}

By Minty's theorem, for every $\epsilon>0$ the resolvent
$J_\epsilon=(\Id+\epsilon A)^{-1}$ is an everywhere-defined,
single-valued, firmly nonexpansive map \cite{r4}. The last property means
\begin{equation*}
     \norm{J_\epsilon z-J_\epsilon w}^2
 \le\langle J_\epsilon z-J_\epsilon w,z-w\rangle
 \text{ for all }z,w\in\R^n.
\end{equation*}
Its isometric leaves are exactly the translates
$\epsilon p+A^{-1}(p)$, as proved in Lemma~\ref{lem:5.2}.
This identifies the geometry, but does not by
itself identify the limiting disintegration: both the conditional
measures and their quotient measures depend on $\epsilon$.
Proposition~\ref{prop:5.4} resolves this difficulty by proving a
total-variation limit on the monotone graph and transferring it to the
conditional measures. Theorem~\ref{thm:5.5} then passes the contraction
inequality to $A^{-1}(p)$ as $\epsilon$ tends to $0$. Localisation and
change of density recover the Hausdorff measure class on convex fibres
and graph faces for arbitrary absolutely continuous ambient measures,
without a curvature assumption.

\subsection{Curvature-dimension condition inheritance and its failure}

The geometric contraction estimate is weaker than the concavity needed
for curvature-dimension condition inheritance. It controls how rapidly the density can
decrease under a homothety, but does not generally control its second
derivatives with the sign required by a curvature-dimension bound.
In codimension one, however, the ordering of disjoint affine hypersurface
sections forces the geometric density to be affine. This gives the
following positive result and the sharp threshold for failure in ambient
dimension.

\begin{maintheorem}[Curvature-dimension condition inheritance and its failure]\label{thm:C}
Let $X\subset\R^n$ be closed and convex with nonempty interior
$\Omega$, and let $d\mu=e^{-\rho}\dd\lambda|_X$ be finite, with
$\rho\in\mathcal C^2(\Omega)$. Suppose that
$(X,\norm{\cdot},\mu)$ satisfies $\CD(\kappa,N)$, where
$\kappa\in\R$ and $N\in[n,\infty]$. For every $1$-Lipschitz map
$u:\R^n\to\R^m$, almost every leaf of codimension one inherits
$\CD(\kappa,N)$. Consequently, almost every leaf inherits this
condition when $n\le2$.

There is a globally firmly nonexpansive map on $\R^3$ and a Euclidean
ball whose Lebesgue conditionals are one-dimensional and fail
$\CD(0,N')$ for every $N'\in(1,\infty]$. There is also a globally
$1$-Lipschitz gradient on $\R^4$, smooth near a Euclidean ball, whose
Lebesgue conditionals are two-dimensional and fail $\CD(0,N')$ for every
$N'\in(2,\infty]$.
\end{maintheorem}

Theorem~\ref{thm:6.2} establishes affinity of the geometric density
without a smoothness assumption on the ambient density.
Corollaries~\ref{cor:6.3} and~\ref{cor:6.4} deduce
curvature-dimension inheritance under the $\mathcal C^2$
assumption above. Both counterexamples have codimension two,
so the ambient dimension threshold is sharp.

The regularity assumption on the ambient density is imposed
only to simplify the exposition. The inheritance conclusions
of Theorem~\ref{thm:C} and Corollaries~\ref{cor:6.3}
and~\ref{cor:6.4} remain valid for finite measures
$d\mu=h\dd\lambda|_X$ with $0<h<\infty$ almost everywhere,
provided that $(X,\norm{\cdot},\mu)$ satisfies $\CD(\kappa,N)$
and the remaining hypotheses are unchanged.
Restriction to closed balls compactly contained in $\Omega$
and \cite[Theorem~3.7]{r22} give a locally Lipschitz
representative of $-\log h$.
The local arguments in
\cite[Theorems~6.11 and~6.20]{MondinoRyborz}, with tests
compactly supported in $\Omega$, give the distributional
Bakry--\'Emery bound when $N>n$. Local convolution and Jensen's inequality
preserve its curvature and dimension parameters; stability
on compact convex leaf sections and exhaustion then give
the asserted inheritance. For $N=n$, one uses the local
density-rigidity argument of \cite[Theorem~3.6]{r22}.
For clarity, we state and prove the inheritance results below
under the $\mathcal C^2$ assumption and omit the details of
this approximation.

The first construction makes the curvature obstruction explicit. In a
unit-speed coordinate along a leaf, the conditional density is
proportional to $s^2+a^2$, where $a>0$ depends on the leaf and the chosen
ball lies in the region $0<s<a$. Hence
\begin{equation*}
    \frac{d^2}{ds^2}\log(s^2+a^2)
 =\frac{2(a^2-s^2)}{(s^2+a^2)^2}>0. 
\end{equation*}
Thus even log-concavity fails, excluding $\CD(0,\infty)$ and all the
finite-dimensional conditions in Theorem~\ref{thm:C}; see
Theorem~\ref{thm:7.1}. The map is vector-valued despite the
one-dimensionality of its leaves, so the example does not contradict
scalar curvature inheritance.

The three-dimensional construction also gives
curvature-dimension failure for fibres of a firmly
nonexpansive map; see Corollary~\ref{cor:monotone-cd-failure}.
Thus the measure-contraction theorem for maximal monotone
relations cannot in general be strengthened to curvature-dimension inheritance.

The gradient counterexample in Theorem~\ref{thm:7.5} exhibits a different
mechanism. Its two-dimensional leaves have orthogonal directions on which
the gradient acts respectively as a translation of the identity and a
reflection. The conditional density can be log-concave separately in
these directions while failing joint log-concavity. The potential is
nonconvex, and the decomposition is into isometric leaves of its gradient,
not fibres of the gradient of a convex function. This distinction is
essential: the two decompositions have different curvature behaviour  \cite{r19}.

The obstruction cannot be removed by allowing a fixed loss in curvature.
Corollary~\ref{cor:7.7} gives, for every $L>0$, rescaled examples on
arbitrarily small balls whose conditionals fail $\CD(-L,N')$ for every
admissible $N'$, including $N'=\infty$. The ambient measure may instead
be fixed as normalised Lebesgue measure on the unit ball or the standard
Gaussian; failure then occurs on a family of leaves of positive quotient
measure. Thus no lower curvature bound depending only on the ambient
curvature and dimension can hold in general. The map may depend on the
prescribed loss $L$; the statement does not assert unbounded curvature
loss for one fixed map.

\subsection{Relation to previous work and localisation}

The disintegration problem originates in Sudakov's approach to the Monge
problem \cite{r33}. Absolute continuity need not pass to a general
partition into segments; see Larman \cite{r25} and the Nikodym set construction
in \cite[Section~2]{r2}. The special geometry of the partition must
therefore enter the proof. Transverse comparison and finite-cone
approximation were developed by Bianchini--Gloyer, Caravenna,
Caravenna--Daneri and Bianchini--Daneri \cite{r7,r10,r11,r6}. The proof of
our transverse estimate follows the approximation method for convex
faces, with the equality defect providing the required separation for
arbitrary nonexpansive maps.

The quantitative results of Caravenna--Daneri go beyond equivalence of
measure classes. Proposition~4.17 and equations~(4.70)--(4.72) in \cite{r11} give local Lipschitz regularity of the
geometric density and endpoint-distance ratio estimates with exponent
$n-k$. After localisation and the choice of continuous representatives,
these estimates imply \eqref{eq:1.3} for convex-gradient fibres.
Multiplication by the ambient density distortion then yields
$\MCP(\kappa,N)$ inheritance in that setting. This is a consequence of
their quantitative theorem, rather than an MCP statement formulated
there. The new geometric scope of the present paper consists of
arbitrary nonexpansive maps on all leaf strata and, through the resolvent
limit, general maximal monotone relations.

For scalar localisation, curvature-dimension inheritance
was established by Klartag \cite{r24}; related synthetic
localisation theorems appear in \cite{r13,r14}.
For vector-valued maps, inheritance on leaves whose
dimension equals the target dimension was proved in
\cite[Theorem~6.2]{r16}. The other dimensional strata
are not covered by that result. In particular,
one-dimensional leaves of a vector-valued map need not
behave like transport rays of a scalar map.

For the measure contraction property,
Cavalletti--Mondino \cite[Theorem~3.6]{r15} prove
$\MCP(\kappa,N)$ inheritance along the transport rays of
any scalar $1$-Lipschitz function on an essentially
nonbranching $\MCP(\kappa,N)$ space. The nonbranching
predecessor is due to Bianchini--Cavalletti
\cite[Theorem~9.5]{r5}. Theorem~\ref{thm:A} extends the
Euclidean conclusion to arbitrary vector-valued maps,
without a restriction on leaf dimension.
Theorem~\ref{thm:C} identifies a further positive case
for curvature-dimension inheritance and shows that
such inheritance fails in general.

For convex-gradient fibres, the measure-contraction conclusion of
Theorem~\ref{thm:B} admits a stronger curvature counterpart. The companion
work \cite{r19} gives polynomial geometric densities with a concave
codimension root and deduces $\CD(\kappa,N)$ inheritance. Its argument
uses cyclic monotonicity and mixed Monge--Amp\`ere measures. No assertion
or proof in the present paper depends on it. Here Corollary \ref{cor:5.7} proves $\MCP(\kappa,N)$ inheritance, and Corollary~\ref{cor:5.9} recovers the measure-class statement on convex
graph faces. 

For localisation applications, the inheritance theorem supplies
geometric control of a given disintegration. Conditional balance remains
a separate requirement: the identities $\int f_i\dd\mu=0$, $i=1,2,\dotsc,m$, do not alone
imply $\int f_i\dd\mu_{\S}=0$ for $i=1,2,\dotsc,m$ on almost every leaf of an optimal $1$-Lipschitz $u$; see \cite[Theorem~5]{r17}, where Klartag's conditional mass-balance conjecture was disproved.

Section~\ref{sec:geometry} recalls leaf geometry and the Euclidean form
of the measure contraction property. Section~\ref{sec:disintegration}
constructs the geometric densities, and Section~\ref{sec:contraction}
proves measure-contraction inheritance and sharpness of the dimension
parameter. Section~\ref{sec:monotone} treats maximal monotone relations
and convex fibres. Section~\ref{sec:curvature} establishes curvature
inheritance in codimension one, and Section~\ref{sec:failure} gives the
counterexamples and the rescaling argument.

\section{Leaf geometry and measure contraction property}\label{sec:geometry}
We follow the notation of \cite{r16}. The Euclidean norm and scalar product are denoted by $\norm{\cdot}$ and $\ip{\cdot}{\cdot}$. For a real-valued function, we identify its differential with its Euclidean gradient. On an affine subspace we make the corresponding intrinsic identification. We write $\lambda$ for Lebesgue measure on $\R^n$, and $\lambda^d$ when the dimension needs to be specified. Hausdorff measure $\Hh^k$ is normalised to agree with Lebesgue measure on a $k$-dimensional affine subspace.

For a convex set $C$, the symbols $\ri C$, $\cl C$ and $\partial C$ denote its relative interior, relative closure and relative boundary. We write $\Conv A$ and $\Aff A$ for the convex and affine hulls of a set $A$. Every affine subspace carries its Euclidean metric. The ball $B(x,r)\subset\mathbb{R}^n$ is open unless stated otherwise, and $\omega_n=\lambda(B(0,1))$. All conditional measures in this article are normalised to have mass one; thus the quotient measure has the same total mass as the original measure.

\subsection{The leaf partition}
Let $u:\R^n\to\R^m$ be $1$-Lipschitz. A \emph{leaf} of $u$ is a
maximal nonempty subset of $\R^n$ on which $u$ preserves distances.
Every leaf $\S$ is closed and convex, and $u|_{\S}$ is affine
\cite[Corollary~2.5]{r16}. If
\begin{equation*}
     V_{\S}=\Aff\S-\Aff\S
\end{equation*}
is its tangent space, there is a unique linear isometry
$T_{\S}:V_{\S}\to\R^m$ such that
\begin{equation}\label{eq:2.1}
 u(x)-u(y)=T_{\S}(x-y)\text{ for all }x,y\in\S.
\end{equation}
We denote the orthogonal projections onto $V_{\S}$ and
$T_{\S}(V_{\S})$ by $P_{\S}$ and $Q_{\S}$, respectively, and write
$P_V$ for the orthogonal projection onto any specified linear subspace
$V$.

Distinct leaves meet only in their relative boundaries, and every
point at which $u$ is differentiable belongs to a unique leaf
\cite[Lemmas~2.11 and~2.12]{r16}. Moreover, whenever two distinct
leaves intersect, their intersection is a face of each
\cite[Lemma~2.2]{r18}.

For completeness, the convexity and affinity assertions follow from
the variance identity. Let $x_1,\ldots,x_\ell$ belong to an isometry
set, let $t_i\ge0$ with $\sum_i t_i=1$, and put
\begin{equation*}
 \bar x=\sum_i t_i x_i,\quad
 \bar u=\sum_i t_i u(x_i).
\end{equation*}
Since the pairwise distances between the $x_i$ are preserved,
the variance identity gives
 \begin{equation*}
      \norm{u(\bar x)-\bar u}^2=\sum_i t_i\norm{u(\bar x)-u(x_i)}^2-\sum_i t_i\norm{\bar x-x_i}^2\leq 0.
 \end{equation*}
Thus $u(\bar x)=\bar u$. Polarisation shows that this affine extension
preserves distances on the convex hull. Continuity then extends the
isometry to its closure, so maximality implies that every leaf is
closed and convex.

Let $N(u)$ denote the Borel set of points at which $u$ is not
differentiable. Rademacher's theorem gives $\lambda(N(u))=0$.
Define
\begin{equation*}
 \S(x)=
 \begin{cases}
  \text{the unique leaf containing }x,&x\notin N(u),\\
  \{x\},&x\in N(u).
 \end{cases}
\end{equation*}
This is a Borel map into the Polish space $\CC(\R^n)$ of nonempty
closed convex subsets of $\R^n$, equipped with the Wijsman topology
\cite[Theorem~5.1 and its proof]{r16}. The singleton values on $N(u)$
are only a convention for defining the map everywhere; they need not
be leaves.

For every finite Borel measure $\mu\ll\lambda$, the disintegration
theorem on standard Borel spaces gives an essentially unique Borel
family of conditional probability measures satisfying
\begin{equation}\label{eq:2.2}
 \nu=\S_\#\mu,\quad
 \mu=\int_{\CC(\R^n)}\mu_{\S}\dd\nu(\S),\quad
 \mu_{\S}(\S)=1,
\end{equation}
where the last identity holds for $\nu$-almost every $\S$
\cite[Theorem~10.4.8 and Corollary~10.4.10]{r8}.
Since $\mu(N(u))=0$, the quotient measure is concentrated on genuine
leaves. Changing the values of the leaf map on $N(u)$ leaves the
quotient measure unchanged and, up to $\nu$-negligible sets, does not
affect the conditional measures.

We shall repeatedly use the nonnegative function
\begin{equation}\label{eq:2.3}
 c_u(x,y)=\norm{x-y}^2-\norm{u(x)-u(y)}^2\ge0,
 \quad x,y\in\R^n.
\end{equation}
Let $C$ be a nonempty convex set on which $u$ is an isometry, and fix
$y\in\R^n$. The function $z\mapsto c_u(z,y)$ is affine on $C$:
the quadratic terms cancel as $u|_C$ is an affine isometry.
Being nonnegative, this function vanishes identically on $C$ whenever
it vanishes at a point of $\ri C$. Consequently,
\begin{equation*}
     x\in\ri C\text{ and } c_u(x,y)=0\text{ imply that } u\text{ is an isometry on }C\cup\{y\}.
\end{equation*}
In particular, maximality gives
\begin{equation*}
     \S=\{y\in\R^n\mid c_u(x,y)=0\}\text{ with }x\in\ri\S.
\end{equation*}
The same identity holds whenever $x$ belongs to a unique leaf $\S$,
even if $x\notin\ri\S$: any two-point isometry set $\{x,y\}$ is
contained in a maximal isometry set, which must then be $\S$.

For $0\le k\le\min\{n,m\}$, let
\begin{equation*}
     \T^k=\{\S\in CC(\mathbb{R}^n)\mid\S\text{ is a leaf of }u,\ \dim\S=k\},
 \quad T_k=\bigcup_{\S\in\T^k}\S.
\end{equation*}
The unions of relative interiors and of relative boundaries of the
leaves in each fixed dimension are Borel
\cite[Corollary~4.7]{r16}; in particular, each $T_k$ is Borel.
The family $\T^k$ is itself a Borel subset of $\CC(\R^n)$.
Indeed, enumerate $\mathbb Q^n$ as $(r_i)_{i\ge1}$ and let $p_i(C)$ be the
metric projection of $r_i$ onto $C\in\CC(\R^n)$. Each map $p_i$ is
continuous for the Wijsman topology, and the points $p_i(C)$ are
dense in $C$. The dimension of $C$ can therefore be determined by
testing affine independence among these points, while isometricity
on $C$ is equivalent to preservation of their pairwise distances.
It follows that
\begin{equation*}
     \mathcal I_k=
 \bigl\{C\in\CC(\R^n)\mid\dim C=k,
 c_u(p_i(C),p_j(C))=0\text{ for all }i,j\ge1\bigr\}
\end{equation*}
is Borel.

For $C\in\mathcal I_k$, choose the first affinely independent
$(k+1)$-tuple from this dense family, using a fixed enumeration of
all such tuples, and denote its barycentre by $b(C)$. This is a
Borel choice, and $b(C)\in\ri C$. The equality-defect argument above
shows that $C$ is a leaf if and only if
\begin{equation*}
    \{y\in\R^n\mid c_u(b(C),y)=0\}\subset C.
\end{equation*}
For $j\ge1$, define the Borel set
\begin{equation*}
     E_j=\bigl\{(C,y)\in\mathcal I_k\times\R^n\mid 
 \norm{y}\le j,\ \dist(y,C)\ge j^{-1},
 c_u(b(C),y)=0\bigr\}.
\end{equation*}
Every section $(E_j)_C$ is compact. The theorem on projections of
Borel sets with $\sigma$-compact sections therefore implies that the projection $\pi(E_j)$ of
$E_j$ onto the first factor is Borel \cite[Theorem~18.18]{r23}. Failure of the preceding
inclusion is equivalent to membership in one of these projections.
Hence
\begin{equation*}
     \T^k=\mathcal I_k\setminus\bigcup_{j\ge1}\pi(E_j)
\end{equation*}
is Borel, as claimed.

\subsection{Measure contraction property}
For $a\in\R$, let
\[
s_a(r)=\begin{cases}
\sin(\sqrt a\,r)/\sqrt a,&a>0,\\
r,&a=0,\\
\sinh(\sqrt{-a}\,r)/\sqrt{-a},&a<0.
\end{cases}
\]
For $\kappa\in\R$, $N>1$ and $0<t\le1$, define
\begin{equation}\label{eq:2.4}
\beta_{\kappa,N}^{(t)}(r)
=t\left(\frac{s_{\kappa/(N-1)}(tr)}{s_{\kappa/(N-1)}(r)}\right)^{N-1},
\qquad \beta_{\kappa,N}^{(t)}(0)=t^N.
\end{equation}
Its domain is $0\le r<R_{\kappa,N}$, where
\[
R_{\kappa,N}=\begin{cases}
\pi\sqrt{(N-1)/\kappa},&\kappa>0,\\
+\infty,&\kappa\le0.
\end{cases}
\]
We use Ohta's normalisation of $\MCP(\kappa,N)$ \cite{r29}. For a finite measure $\eta$ with closed convex support $Y$ in a Euclidean affine space, uniqueness of geodesics identifies this condition with
\begin{equation}\label{eq:2.5}
\eta(H_{o,t}(B))\ge\int_B\beta_{\kappa,N}^{(t)}(\norm{x-o})\dd\eta(x),
\quad H_{o,t}(\cdot)=o+t(\cdot-o),
\end{equation}
for all $o\in Y$, $0<t\le1$, and Borel $B\subset Y\cap B(o,R_{\kappa,N})$. Equivalently,
\begin{equation}\label{eq:2.6}
(H_{o,t})_\#\bigl(\mathbf{1}_{B(o,R_{\kappa,N})}\beta_{\kappa,N}^{(t)}(\norm{\cdot-o})\eta\bigr)\le\eta.
\end{equation}
The equivalence follows by taking inverse images under $H_{o,t}$, which is injective, see \cite[Definition~2.1 and Lemma~2.3]{r29}. For a locally finite measure with full support, the same characterisation holds by first considering bounded source sets and then using monotone convergence. When $\kappa\le0$, the coefficient is bounded by one. When $\kappa>0$, the restriction to $B(o,R_{\kappa,N})$ is part of the condition and will be retained throughout the proofs.

\begin{lemma}[Pointwise ambient distortion]\label{lem:2.1}
Let $X\subset\R^n$ be closed and convex with nonempty interior $\Omega$, and let $d\mu=h\,d\lambda|_X$ satisfy $\MCP(\kappa,N)$, where $0<h<\infty$ almost everywhere in $\Omega$. Let $L_h\subset\Omega$ be the set of Lebesgue points of $h$ with positive finite value. Then $\lambda(\Omega\setminus L_h)=0$; we use the precise representative of $h$ on $L_h$. If $o\in X$, $0<t\le1$, $\norm{x-o}<R_{\kappa,N}$ and $x,H_{o,t}(x)\in L_h$, then
\begin{equation}\label{eq:2.7}
t^nh(H_{o,t}(x))\ge\beta_{\kappa,N}^{(t)}(\norm{x-o})h(x).
\end{equation}
\end{lemma}
\begin{proof}
For sufficiently small $\epsilon>0$, the ball $B(x,\epsilon)$ lies in $\Omega\cap B(o,R_{\kappa,N})$. Its homothetic image is $B(H_{o,t}(x),t\epsilon)$. Applying \eqref{eq:2.5} to this ball and dividing by $\omega_n\epsilon^n$, we obtain an inequality between averages of $h$. As $\epsilon\downarrow0$, Lebesgue differentiation at $x$ and $H_{o,t}(x)$, together with continuity of the coefficient, gives \eqref{eq:2.7}.
\end{proof}

\begin{lemma}[Restriction and change of density]\label{lem:2.2}
Let $\gamma=\int\gamma_q\dd\vartheta(q)$ be a normalised disintegration over a Borel map $\pi$, and let $a\ge0$ be Borel with $\int a\dd\gamma<\infty$. Put $A(q)=\int a\dd\gamma_q$. The quotient measure of $a\,d\gamma$ is $A\,d\vartheta$, and its conditional on $\{q\mid A(q)>0\}$ is $A(q)^{-1}a\, d\gamma_q$.
\end{lemma}
\begin{proof}
Integrating a nonnegative Borel function against the proposed disintegration and applying Tonelli's theorem gives its integral against $a\,d\gamma$. The assertion follows from uniqueness of the normalised conditional measures.
\end{proof}

\section{Geometric disintegration}\label{sec:disintegration}
We assume throughout this section that $u:\R^n\to\R^m$ is $1$-Lipschitz. Let us recall that $\S$ denotes the Borel leaf map and $\T^k$ the family of leaves of dimension $k$. If $x\notin N(u)$, then $\S(x)$ is the unique leaf containing $x$. We denote the tangent space of a leaf $\S$ by $V_{\S}$. Let $w$ be a smooth positive probability density on $\R^n$, and let
\begin{equation*}
d\sigma=w\,d\lambda,\quad\nu_0=\S_\#\sigma,\qquad
\sigma=\int\sigma_{\S}\dd\nu_0(\S).
\end{equation*}
All conditional measures are normalised to have mass one. The density $w$ plays an auxiliary role.

\begin{theorem}[Geometric disintegration]\label{thm:3.1}
For $\nu_0$-almost every leaf $\S$ of dimension $k\ge1$, there is a positive function $g_{\S}$ on $\ri\S$ such that $\log g_{\S}\in\Lip_{\mathrm{loc}}(\ri\S)$,
\begin{equation}\label{eq:3.1}
\int_{\S}wg_{\S}\dd\Hh^k=1,
\end{equation}
and
\begin{equation}\label{eq:3.2}
g_{\S}(o+t(x-o))\ge t^{n-k}g_{\S}(x)\text{ for all }o\in\S,\ x\in\ri\S,\ 0<t\le1.
\end{equation}
The functions may be chosen measurably in the leaf variable. For every finite measure $d\mu=h\,d\lambda$, with $h\ge0$ Borel, put
\begin{equation*}
   m_{\S}=\int_{\S}hg_{\S}\dd\Hh^k. 
\end{equation*}
Then $d\S_\#\mu=m_{\S}\dd\nu_0$, and the conditional measures of $\mu$ satisfy
\begin{equation}\label{eq:3.3}
d\mu_{\S}=\frac{hg_{\S}}{m_{\S}}\dd\Hh^k|_\S
\quad\text{for }m_{\S}\nu_0\text{-almost every }\S.
\end{equation}
Here $0<m_{\S}<\infty$ for $m_{\S}\nu_0$-almost every leaf. On a zero-dimensional leaf $\S=\{x\}$, the same statements hold with $g_{\S}(x)=1/w(x)$ and $\Hh^0|_\S$. The values of $g_{\S}$ on relative boundaries are negligible. A different choice of $w$ multiplies $g_{\S}$ by a positive constant depending only on $\S$.
\end{theorem}
The exponent in \eqref{eq:3.2} is the codimension of the leaf. We shall prove the theorem without any curvature assumption. We first establish a comparison of transverse slices. We then use this comparison to show that the relative boundaries have Lebesgue measure zero and to construct the conditional densities.

\subsection{Comparison of transverse slices}
Recall the equality defect $c_u$ from \eqref{eq:2.3}. If $x$ belongs to a unique leaf $\S$, or if $x\in\ri\S$, then $c_u(x,y)=0$ implies $y\in\S$. The next lemma converts this qualitative separation of leaves into a measure estimate. Its exponent is the dimension of the transverse slice. The approximation follows the method for convex faces in \cite[Lemmas~4.7 and~4.14]{r11}; here the equality defect supplies the separation for arbitrary nonexpansive maps.

\begin{lemma}[Transverse comparison]\label{lem:3.2}
Let $M$ be a $(d+1)$-dimensional affine subspace, let $\ell:M\to\R$ be affine with unit linear part, and let $A\subset\ell^{-1}(0)$ be bounded and Borel. Suppose that $\tau:A\to \ell^{-1}(H)$ is bounded and Borel, where $H>0$, and that
\begin{enumerate}
\item $c_u(x,\tau(x))=0$ for every $x\in A$;
\item if $y\in\cl \tau(A)$ and $c_u(x,y)=0$, then $y=\tau(x)$;
\item $c_u(y,y')>0$ for distinct $y,y'\in\tau(A)$.
\end{enumerate}
For $0\le s<1$, write $\tau_s=(1-s)\Id+s\tau$. Then
\begin{equation}\label{eq:3.4}
(\tau_s)_\#(\Hh^d|_ A)\le(1-s)^{-d}\Hh^d|_{\ell^{-1}(sH)}.
\end{equation}
\end{lemma}
\begin{proof}
For each $j\in\mathbb{N}$ let $Y_j\subset\tau(A)$ be a finite $\delta_j$-net, where $\delta_j\downarrow0$. Using a fixed ordering of $Y_j$, choose
\begin{equation*}
\tau_j(x)\in\argmin\{c_u(x,y)\mid y\in Y_j\}
\end{equation*}
for each $x\in A$. The maps $\tau_j$ are Borel. By uniform continuity of $c_u$ on a compact set containing all the pairs under consideration, and the assumption that $c_u(x,\tau(x))=0$ we get that $\lim_{j\to\infty}c_u(x,\tau_j(x))=0$ for every $x\in A$. Every cluster point of $\tau_j(x)$ belongs to $\cl\tau(A)$; the second assumption therefore implies $\lim_{j\to\infty}\tau_j(x)=\tau(x)$.

If $(a,b)$ and $(a',b')$ belong to the graph of $\tau_j$, then
\begin{equation*}
c_u(a,b)+c_u(a',b')\le c_u(a,b')+c_u(a',b).
\end{equation*}
If the corresponding segments $[a,b], [a',b']$ intersect at an interior point of both of the segments and $b\ne b'$, then $a-a'=-r(b-b')$ for some $r>0$. Expanding the preceding inequality and using the Lipschitz bound on $u$, we obtain
\begin{equation*} 
r\norm{b-b'}^2\le-\ip{u(a)-u(a')}{u(b)-u(b')}
\le\norm{a-a'}\,\norm{b-b'}=r\norm{b-b'}^2.
\end{equation*}
Equality implies $\norm{u(b)-u(b')}=\norm{b-b'}$, contrary to the third assumption. Distinct inverse images of $\tau_j$ are disjoint at level $\ell^{-1}(0)$, and the above argument shows that the interpolating segments
associated with distinct target points do not meet at
intermediate levels. Thus the images under $(1-s)\Id+s\tau_j$ of distinct inverse images of $\tau_j$ are disjoint. On each such inverse image the map is a homothety of ratio $1-s$. A simple form of the area formula \cite{r21} yields \eqref{eq:3.4} with $\tau_j$ in place of $\tau$. Since $\Hh^d(A)<\infty$, pointwise convergence and boundedness of the images give weak convergence of the pushforwards. Testing against nonnegative continuous functions of compact
support and passing to the limit yields \eqref{eq:3.4}.
\end{proof}

\subsection{Measurable charts and relative boundaries}
We first construct countably many affine charts for almost every positive-dimensional leaf. We then use the transverse comparison to prove that the union of the relative boundaries is Lebesgue-negligible, by contracting boundary points towards interior anchors; compare \cite[Lemma~4.19]{r11} for a similar statement about  faces of convex functions.

\begin{proposition}[Countable affine charts]\label{prop:3.4}
For each $1\le k\le\min(n,m)$ there are a Borel subfamily $\T^{k,0}\subset\T^k$, with $\nu_0(\T^k\setminus\T^{k,0})=0$, and countably many charts
\begin{equation}\label{eq:3.7}
\Phi:A\times D\longrightarrow\R^n,\qquad
\Phi(a,q)=a+L_a(q-q_*),
\end{equation}
with the following properties. The space $V$ is a $k$-dimensional linear space, and $C,D\subset V$ are bounded open convex polytopes satisfying $\cl C\subset D$ and $q_*\in C$. The set $A\subset P_V^{-1}(q_*)$ is compact, and $a\mapsto L_a$ is continuous on $A$. There is a Borel injection $A\ni a\mapsto\S_a\in\T^k$ such that
\begin{equation*}
    L_a=(P_V|_{V_{\S_a}})^{-1},\quad
\Phi(\{a\}\times D)\subset\ri\S_a,\quad
P_V\Phi(a,q)=q.
\end{equation*}
The map $\Phi$ is injective. For every $\S\in\T^{k,0}$ and every $x,y\in\ri\S$, there are a chart and an anchor $a\in A$ such that $\S_a=\S$ and $x,y\in\Phi(\{a\}\times C)$.
\end{proposition}
\begin{proof}
We shall first show that $\S\mapsto V_{\S}$ is Borel. Let $(r_j)$ be countable and dense in $\R^n$, and let $x_j(\S)$ be the metric projection of $r_j$ onto $\S$. Each $x_j$ is continuous for the Wijsman topology. Indeed, a convergent sequence of sets has a bounded set of projections, and every convergent subsequence of these projections has its limit in $\S$ at distance $\dist(r_j,\S)$ from $r_j$. This determines the limit uniquely. The points $x_j(\S)$ are dense in $\S$. Choose the first $(k+1)$-tuple that is affinely independent. This is a Borel map and the span of its differences is $V_{\S}$.

Choose a countable, Borel partition of the $k$-dimensional Grassmannian into sets $\mathcal U(V)$ such that
\begin{equation*}
\norm{(P_V|_W)^{-1}}\le2\text{ for all }W\in \mathcal U(V).
\end{equation*}
For $V_{\S}\in \mathcal U(V)$, set $L_{\S}=(P_V|_{V_{\S}})^{-1}$. Given $q_*\in V$, the unique point of $\Aff\S\cap P_V^{-1}(q_*)$ is
\begin{equation*}
a(\S)=x_1(\S)+L_{\S}(q_*-P_Vx_1(\S)).
\end{equation*}
Fix orthonormal coordinates on each $V$. Let $C,D$ range over rational -- with respect to the chosen coordinates -- polytopes with $\cl C\subset D$, let $q_*$ range over a fixed countable dense subset of $C$, and let $r\in\N$. Consider the Borel family $G$ of leaves with $V_{\S}\in \mathcal U(V)$ such that
\begin{equation}\label{eq:3.8}
a(\S)+L_{\S}(q-q_*)\in\S\text{ for all }q\in\cl D+\cl B_V(0,1/r).
\end{equation}

Set $K=\cl D+\cl B_V(0,1/r)$ and choose a countable
dense subset $(q_j)$ of $K$. Write
\begin{equation*}
 F_{\S}(q)=a(\S)+L_{\S}(q-q_*).
\end{equation*}
Since $F_{\S}$ is continuous and $\S$ is closed, condition
\eqref{eq:3.8} is equivalent to
\begin{equation*}
    \dist(F_{\S}(q_j),\S)=0
 \qquad\text{for every }j\ge1.
\end{equation*}

For each fixed $j$, the left-hand side is a Borel function of $\S$:
the maps $\S\mapsto a(\S)$ and $\S\mapsto L_{\S}$ are Borel, and
$(y,\S)\mapsto\dist(y,\S)$ is jointly continuous for the Wijsman
topology. Thus $G$ is a Borel subfamily of $\T^k$.

For $\S\in G$ and $q\in\cl D$, condition \eqref{eq:3.8}
also gives
\begin{equation*}
 F_{\S}(q)+L_{\S}\bigl(B_V(0,1/r)\bigr)\subseteq\S.
\end{equation*}

The set on the left is a neighbourhood of $F_{\S}(q)$ relative to
$\Aff\S$, because $L_{\S}:V\to V_{\S}$ is a linear isomorphism.
Consequently, $F_{\S}(\cl D)\subseteq\ri\S$ and, in
particular, $a(\S)=F_{\S}(q_*)\in\ri\S$.
If two leaves in $G$ have the same anchor, their relative interiors
therefore intersect, so the leaves coincide. Hence the anchor map
is injective on $G$.

The Lusin--Souslin theorem \cite{r23}, applied to the Borel
injection
\begin{equation*}
 a|_G:G\to P_V^{-1}(q_*),
\end{equation*}
now shows that $a(G)$ is Borel and that the inverse map
$a\mapsto\S_a$ is Borel.

Consider the image of $\nu_0|_G$ under the anchor map $a$. By Lusin's theorem and inner regularity, there are countably many compact sets $A\subset a(G)$, covering $a(G)$ up to a negligible set, on which $a\mapsto L_{\S_a}$ is continuous. For each such $A$, define \eqref{eq:3.7}. If two points in the image of  $\Phi$ agree, the corresponding leaves intersect in their relative interia, and thus they coincide. Thus $\Phi$ is injective and its image is Borel.

There are countably many choices of the data. For each choice, remove the $\nu_0$-negligible family not covered by the corresponding compact sets. Denote the remaining family by $\T^{k,0}$. We keep the compact anchor sets unchanged, even when some of their leaves have been removed from $\T^{k,0}$ in this procedure.

Finally, $P_V(\ri\S)$ is open and convex. Any two of its points lie in a rational polytope $C$ with $\cl C\subset D$ and $\cl D\subset P_V(\ri\S)$. Choose $q_*$ as above. Compactness gives $1/r>0$ for which \eqref{eq:3.8} holds. Every leaf in $\T^{k,0}$ therefore occurs in a compact chart for these data. This proves the assertion concerning pairs of points.
\end{proof}

\begin{lemma}[Negligibility of relative boundaries]\label{lem:3.3}
For every $k\ge1$, the union of the relative boundaries of all the
$k$-dimensional leaves is $\lambda$-negligible.
\end{lemma}

\begin{proof}
It suffices to consider $1\le k\le\min\{n,m\}$. Let $ \partial T_k$ be the union of all relative boundaries of $k$-dimensional leaves.
This set is Borel. Since $\nu_0=\S_\#(w\lambda)$ and $w>0$,
Proposition~\ref{prop:3.4} gives
\begin{equation*}
     \lambda\bigl(\S^{-1}(\T^k\setminus\T^{k,0})\bigr)=0.
\end{equation*}
For each chart, the family
\begin{equation*}
    \mathcal G_A=\{\S_a:a\in A\}
\end{equation*}
is Borel, and the inverse map $\mathcal G_A\ni\S_a\mapsto a$ is
Borel, by the Lusin--Souslin theorem \cite{r23}. These countably many
families cover $\T^{k,0}$. As $N(u)$ is negligible, it therefore suffices
to fix one chart and prove that the Borel set
\begin{equation*}
    E=(\partial T_k\setminus N(u))\cap\S^{-1}(\mathcal G_A)
\end{equation*}
is negligible.

Suppose that $\lambda(E)>0$, and choose a compact set $K\subset E$
with $\lambda(K)>0$. For $x\in K$, let $b(x)\in A$ be the anchor
of $\S(x)$. The preceding Borel inverse defines a Borel map
$b:K\to A$, and $b(x)\in\ri\S(x)$.
For $0<t<1$, set
\begin{equation*}
    F_t(x)=(1-t)x+tb(x)\text{ for } x\in K.
\end{equation*}
Convexity gives $F_t(x)\in\ri\S(x)$. If $F_t(x)=F_t(y)$,
the corresponding leaves have intersecting relative interiors and
therefore coincide. Their anchors agree, so $(1-t)(x-y)=0$ and
$x=y$. Thus $F_t$ is a Borel injection, and $F_t(K)$ is Borel
by the Lusin--Souslin theorem \cite{r23}.
Since every point of $K$ lies on the relative boundary of its unique
leaf, $F_t(K)\cap K=\emptyset$. Boundedness of $K$ and $A$
also gives a constant $M<\infty$, independent of $t$, such that
\begin{equation}\label{eq:3.5}
 \sup\{\norm{F_t(x)-x}\mid x\in K\}\le Mt.
\end{equation}

Put $d=n-k$ and $K_q=K\cap P_V^{-1}(q)$ for $q\in V$.
Notice that $K_{q_*}=\emptyset$: a point of $\S(x)$ projecting
to $q_*$ must equal its interior anchor $b(x)$.
For $q\ne q_*$ with $K_q\ne\emptyset$, apply
Lemma~\ref{lem:3.2} to $K_q$ and $b|_{K_q}$ in the affine space
\begin{equation*}
     M_q=P_V^{-1}\bigl(q+\mathbb R(q_*-q)\bigr),
 \quad
 \ell_q(z)=\left\langle P_Vz-q,
       \frac{q_*-q}{\norm{q_*-q}}\right\rangle.
\end{equation*}
The map $b|_{K_q}$ is bounded and Borel, with image at level
$H_q=\norm{q_*-q}>0$, and $c_u(x,b(x))=0$.
If $x\in K_q$ and $a\in\cl(b(K_q))$ satisfy $c_u(x,a)=0$,
then $a\in\S(x)$ by uniqueness of the leaf through $x$.
Since $P_Va=q_*=P_Vb(x)$ and $P_V$ is injective on
$\Aff\S(x)$, it follows that $a=b(x)$.
Finally, distinct anchors lie in the relative interiors of distinct
leaves and hence have strictly positive equality defect.
Thus all assumptions of Lemma \ref{lem:3.2} hold, and
\begin{equation*}
    (F_t)_\#\bigl(\Hh^d|_{K_q}\bigr)
 \le (1-t)^{-d}\Hh^d|_{
       P_V^{-1}\bigl((1-t)q+tq_*\bigr)}.
\end{equation*}
Integrating in $q$ by Fubini's theorem and making the change of
variables $q'=(1-t)q+tq_*$ yields
\begin{equation}\label{eq:3.6}
 (F_t)_\#(\lambda|_K)\le(1-t)^{-n}\lambda.
\end{equation}
Here the transverse comparison contributes $(1-t)^{-(n-k)}$
and the change of variables contributes $(1-t)^{-k}$.

Applying \eqref{eq:3.6} to the Borel set $F_t(K)$ and using
\eqref{eq:3.5}, we obtain
\begin{equation*}
     (1-t)^n\lambda(K)
 \le\lambda(F_t(K))
 \le\lambda\bigl(\{z:\dist(z,K)\le Mt\}\setminus K\bigr).
\end{equation*}
The last term converges to zero as $t$ converges to zero, by compactness of $K$ and continuity
from above of Lebesgue measure on its bounded neighbourhoods.
This contradicts $\lambda(K)>0$. Therefore $\lambda(E)=0$,
and the countable chart cover proves the assertion.
\end{proof}

\subsection{The local conditional density}
\begin{lemma}[Product density in an affine chart]\label{lem:3.5}
Fix a chart from Proposition~\ref{prop:3.4}, and put
\begin{equation*}
     Z=\Phi(A\times C),\quad
 \Sigma_q=\Phi(A\times\{q\}),\quad q\in D,\quad d=n-k.
\end{equation*}
If $\lambda(Z)>0$, then $\zeta=\Hh^d|_A$ is finite and nonzero,
and there is a Borel function $f:A\times C\to(0,\infty)$, bounded
above and bounded away from zero, such that
\begin{equation}\label{eq:3.9}
 d(\Phi^{-1})_\#(\lambda|_Z)
 =f\,d\bigl(\zeta\otimes(\lambda^k|_C)\bigr).
\end{equation}
For $\zeta$-almost every $a$, the function $\log f(a,\cdot)$ is
Lipschitz on $C$, and
\begin{equation}\label{eq:3.10}
 f(a,(1-t)q_0+tq_1)\ge t^d f(a,q_1)
 \text{ for all }q_0,q_1\in C,\ 0<t\le1.
\end{equation}
Moreover, put
\begin{equation*}
     W(a)=\int_C w(\Phi(a,q))f(a,q)\dd \lambda^k(q),
 \quad J_a=\sqrt{\det(L_a^*L_a)}.
\end{equation*}
Then $0<W(a)<\infty$ and the conditional measure at $a$ of
$\sigma|_Z$ has $\Hh^k$-density
\begin{equation}\label{eq:3.11}
 \frac{w(\Phi(a,q))f(a,q)}{W(a)J_a}.
\end{equation}
For $W\zeta$-almost every $a$, this conditional agrees with
$\sigma_{\S_a}|_Z/\sigma_{\S_a}(Z)$.
\end{lemma}

\begin{proof}
We first compare the transverse slice measures, then choose
their Radon--Nikodym densities with continuous dependence on
the leaf coordinate, and finally identify the resulting
conditional measures.

\textbf{Step 1: comparison of the transverse measures.}

For $q\in D$, let $\Phi_q(a)=\Phi(a,q)$ and define
\begin{equation*}
    M_q=(\Phi_q^{-1})_\#(\Hh^d|_{\Sigma_q}).
\end{equation*}
These measures are finite and form a Borel kernel. For every Borel set $E\subset A$, the set
$Z_E=\Phi(E\times D)$ is Borel by the Lusin--Souslin theorem
\cite{r23}, since $\Phi$ is a continuous injection.
The identity $P_V\Phi(a,q)=q$ implies
\begin{equation*}
     Z_E\cap P_V^{-1}(q)=\Phi_q(E)\text{ for }q\in D.
\end{equation*}
Since $P_V^{-1}(q)=q+V^\perp$ and translation by $q$ preserves
$\Hh^d$, we obtain
\begin{equation*}
    M_q(E)=\Hh^d(\Phi_q(E))
       =\int_{V^\perp}\mathbf{1}_{Z_E}(q+z)\dd\Hh^d(z).
\end{equation*}
The integrand is jointly Borel in $(q,z)$, and
$\Hh^d|_{V^\perp}$ is $\sigma$-finite. Thus $q\mapsto M_q(E)$
is Borel.

Fix distinct $p_0,p_1\in D$, and put
\begin{equation*}
 H=\norm{p_1-p_0}
 \text{ and }
 e=\frac{p_1-p_0}{H}.
\end{equation*}
Consider the $(d+1)$-dimensional affine space
$M=P_V^{-1}(p_0+\R e)$, equipped with the affine coordinate
\begin{equation*}
 \ell(z)=\ip{P_Vz-p_0}{e}
 \text{ for }z\in M.
\end{equation*}
The linear part of $\ell$ has norm one, and its level sets in
$ M$ satisfy
\begin{equation*}
 \ell^{-1}(h)=P_V^{-1}(p_0+he)
 \text{ for }h\in\R.
\end{equation*}
In particular, $\Sigma_{p_0}\subset\ell^{-1}(0)$ and
$\Sigma_{p_1}\subset\ell^{-1}(H)$.
Apply Lemma~\ref{lem:3.2} with source $\Sigma_{p_0}$ and
correspondence $\tau=\Phi_{p_1}\circ\Phi_{p_0}^{-1}$, so that
\begin{equation*}
 \tau(\Phi_{p_0}(a))=\Phi_{p_1}(a)
 \text{ for }a\in A.
\end{equation*}
Each $\Phi_p$ is a continuous injection from the compact set $A$
and hence a homeomorphism onto the compact set $\Sigma_p$.
Thus the source is bounded and Borel, and $\tau$ is continuous
and bounded.

Corresponding points belong to the same leaf, so
$c_u(x,\tau(x))=0$ for every $x\in\Sigma_{p_0}$.
To check the closure condition, let $x=\Phi_{p_0}(a)$ and
$y\in\cl(\tau(\Sigma_{p_0}))=\Sigma_{p_1}$ satisfy
$c_u(x,y)=0$. Since $x\in\ri\S_a$, the equality-defect property
gives $y\in\S_a$. Both $y$ and $\tau(x)$ project to $p_1$,
and $P_V$ is injective on $\Aff\S_a$, so $y=\tau(x)$.
Likewise, if two target points have zero equality defect, then
they belong to the same leaf, since each lies in the relative
interior of its leaf. Their equal projections force them to
coincide. Hence distinct target points have strictly positive
equality defect, and all hypotheses of Lemma~\ref{lem:3.2} hold.

For $0\le s<1$, write $p_s=(1-s)p_0+sp_1\in D$ and
$\tau_s=(1-s)\Id+s\tau$. Affinity in the second chart coordinate
gives
\begin{equation*}
 \tau_s(\Phi_{p_0}(a))=\Phi_{p_s}(a)
 \text{ for }a\in A.
\end{equation*}
The lemma therefore yields
\begin{equation*}
 (\tau_s)_\#(\Hh^d|_{\Sigma_{p_0}})
 \le (1-s)^{-d}\Hh^d|_{P_V^{-1}(p_s)}.
\end{equation*}
For every Borel set $E\subset A$, the set $\Phi_{p_s}(E)$ is
Borel, and injectivity of $\Phi_{p_s}$ gives
\begin{equation*}
 \tau_s^{-1}(\Phi_{p_s}(E))=\Phi_{p_0}(E).
\end{equation*}
Evaluating the preceding domination on $\Phi_{p_s}(E)$, we obtain
\begin{equation*}
 \begin{aligned}
 M_{p_0}(E)
 &= (\tau_s)_\#(\Hh^d|_{\Sigma_{p_0}})
       \bigl(\Phi_{p_s}(E)\bigr)\\
 &\le (1-s)^{-d}\Hh^d\bigl(\Phi_{p_s}(E)\bigr)\\
 &= (1-s)^{-d}M_{p_s}(E).
 \end{aligned}
\end{equation*}
Thus, as measures on $A$,
\begin{equation*}
 M_{p_0}\le(1-s)^{-d}M_{(1-s)p_0+sp_1}
 \text{ for }0\le s<1.
\end{equation*}

Set $\rho=\tfrac12\dist(\cl C,V\setminus D)>0$.
For distinct $q,q'\in C$, write $r=\norm{q-q'}$ and put
\begin{equation*}
 q''=q'+\frac{\rho}{r}(q'-q)\in D.
\end{equation*}
Then
\begin{equation*}
 q'=\frac{\rho}{r+\rho}q+\frac{r}{r+\rho}q''.
\end{equation*}
Apply the above comparison with $p_0=q$, $p_1=q''$ and
$s=r/(r+\rho)$. The intermediate coordinate is $q'$, and
$1-s=\rho/(r+\rho)$, so
\begin{equation*}
 M_q\le\left(1+\frac r\rho\right)^d M_{q'}.
\end{equation*}
For the reverse comparison, interchange $q$ and $q'$ and use the
new target coordinate
\begin{equation*}
 q^-=q+\frac{\rho}{r}(q-q')\in D.
\end{equation*}
The same argument gives
\begin{equation*}
 M_{q'}\le\left(1+\frac r\rho\right)^dM_q.
\end{equation*}
Consequently,
\begin{equation*}
 \left(1+\frac{\norm{q-q'}}{\rho}\right)^{-d}M_q
 \le M_{q'}\le
 \left(1+\frac{\norm{q-q'}}{\rho}\right)^dM_q
 \text{ for }q,q'\in C,
\end{equation*}
where the case $q=q'$ is immediate.

Finally, set $q_t=(1-t)q_0+tq_1$.
For distinct $q_0,q_1\in C$ and $0<t<1$, apply the same comparison
with $p_0=q_1$, $p_1=q_0$ and $s=1-t$.
The intermediate coordinate is $p_s=q_t$, and $1-s=t$.
Hence
\begin{equation*}
 M_{q_1}\le t^{-d}M_{q_t}
 \text{ for }q_0,q_1\in C,\ 0<t\le1,
\end{equation*}
where the cases $q_0=q_1$ and $t=1$ are immediate.

\textbf{Step 2: construction of a continuous density.}

Since $\Phi_{q_*}=\Id$ on $A$, we have
$M_{q_*}=\zeta=\Hh^d|_A$. This measure is nonzero, since otherwise
the comparison and Fubini's theorem would give $\lambda(Z)=0$.
All $M_q$, $q\in C$, are uniformly equivalent to $\zeta$.
Choose a countable dense set $Q\subset C$ containing $q_*$ and
closed under rational convex combinations. For $q\in Q$, choose
a Borel representative $f_q=dM_q/d\zeta$.
After deleting one Borel $\zeta$-negligible subset of $A$, the preceding
measure inequalities imply simultaneously that
\begin{equation*}
 \begin{split}
 f_{q_*}(a)&=1,\\
 |\log f_q(a)-\log f_{q'}(a)|
 &\le d\log\left(1+\frac{\norm{q-q'}}{\rho}\right)
 \le\frac d\rho\norm{q-q'},\\
 f_{(1-t)q_0+tq_1}(a)&\ge t^d f_{q_1}(a)
 \end{split}
\end{equation*}
for all $q,q',q_0,q_1\in Q$ and rational $t\in(0,1)$.
Only countably many inequalities are involved.

For each remaining $a$, extend $q\mapsto\log f_q(a)$ continuously
from $Q$ to $C$, and let $f(a,q)$ be its exponential. Put $f=1$
on the exceptional set of anchors. The function $f$ is jointly
Borel: choose Borel maps $q_j:C\to Q$ with
$\norm{q_j(q)-q}<1/j$ and use the pointwise limit of
$f_{q_j(q)}(a)$ off that exceptional set.
The comparison with $q_*$ gives uniform positive lower and finite
upper bounds. Continuity gives \eqref{eq:3.10} for every pair of
points and every time.

If $q_j\in Q$ converges to $q\in C$, the measure comparison gives
\begin{equation*}
    \lim_{j\to\infty}M_{q_j}(E)= M_q(E)
\end{equation*}
for every Borel $E\subset A$.
Bounded convergence therefore yields
\begin{equation*}
     M_q(E)=\lim_{j\to\infty}\int_E f(a,q_j)\dd\zeta(a)
       =\int_E f(a,q)\dd\zeta(a).
\end{equation*}
Thus $dM_q=f(\cdot,q)d\zeta$ for every $q\in C$, and Fubini's
theorem proves \eqref{eq:3.9}.

\textbf{Step 3: identification of the conditional measures.}

Since $\sigma=w\lambda$, equation~\eqref{eq:3.9} gives
\begin{equation*}
d (\Phi^{-1})_\#(\sigma|_Z)
 =(w\circ\Phi)f\,d\bigl(\zeta\otimes(\lambda^k|_C)\bigr).
\end{equation*}
Integrating this joint density in $q$ gives the marginal $W\dd\zeta$
on the anchor set $A$, where
\begin{equation*}
    W(a)=\int_C w(\Phi(a,q))f(a,q)\dd \lambda^k(q).
\end{equation*}
Dividing by this integral therefore gives the conditional
probability measure in the coordinate variable:
\begin{equation*}
 \frac{w(\Phi(a,q))f(a,q)}{W(a)}\dd \lambda^k(q).
\end{equation*}
This is the normalisation described in Lemma~\ref{lem:2.2}.
Positivity and continuity of $w$ on the compact set
$\Phi(A\times\cl C)$, together with the positive lower and finite
upper bounds for $f$, imply $0<W(a)<\infty$.
To express the conditional on the leaf section rather than in
coordinates, fix $a$ and push the measure forward by $q\mapsto\Phi(a,q)$.
This affine map has constant $k$-dimensional Jacobian
$J_a=\sqrt{\det(L_a^*L_a)}$, so the area formula gives
\begin{equation*}
 (\Phi(a,\cdot))_\#(\lambda^k|_C)
 =J_a^{-1}\Hh^k|_{\Phi(\{a\}\times C)}.
\end{equation*}
Thus the density with respect to $\Hh^k$ is the coordinate density
divided by $J_a$, which proves \eqref{eq:3.11}.
We now identify the chart conditionals with the restrictions of the
global leaf conditionals. Let $\beta_a$ denote the probability measure
on $\Phi(\{a\}\times C)$ with density \eqref{eq:3.11}.
Set $G_A=\{\S_a:a\in A\}$ and let $\iota:G_A\to A$ be the
Borel isomorphism $\iota(\S_a)=a$.
The anchor map $\alpha:Z\to A$, defined by
$\alpha(\Phi(a,q))=a$, satisfies
\begin{equation*}
 \alpha(x)=\iota(\S(x))
 \text{ for }x\in Z\setminus N(u).
\end{equation*}
Thus the anchor and leaf labels describe the same partition of $Z$,
up to a $\sigma$-negligible set.

By Lemma~\ref{lem:2.2}, restricting the global disintegration to $Z$
gives the quotient measure
\begin{equation*}
 d\nu_Z=\sigma_{\S}(Z)\dd\nu_0
\end{equation*}
and the normalised conditionals
$\sigma_{\S}|_Z/\sigma_{\S}(Z)$ wherever $\sigma_{\S}(Z)>0$.
The measure $\nu_Z$ is concentrated on $G_A$.
Since $\sigma(N(u))=0$, the two descriptions of the quotient give
\begin{equation}\label{eq:3.12}
 W\dd\zeta=d\alpha_\#(\sigma|_Z)=\dd\iota_\#\nu_Z.
\end{equation}
Uniqueness of disintegration \cite{r8} therefore yields
\begin{equation*}
 \beta_a=\frac{\sigma_{\S_a}|_Z}{\sigma_{\S_a}(Z)}
 \text{ for }W\zeta\text{-almost every }a,
\end{equation*}
with a positive denominator at almost every such anchor.

Finally, if $E\subset A$ is Borel and $\zeta(E)=0$, then
\eqref{eq:3.12} gives
\begin{equation*}
 \int_{\iota^{-1}(E)}\sigma_{\S}(Z)\dd\nu_0(\S)
 =\int_E W(a)\dd\zeta(a)=0.
\end{equation*}
Consequently,
\begin{equation*}
 \nu_0\bigl(\{\S\in\iota^{-1}(E)\mid \sigma_{\S}(Z)>0\}\bigr)=0.
\end{equation*}
This transfers the exceptional anchors 
to a negligible family of leaves whenever the chart has positive
conditional mass.
\end{proof}

\begin{proof}[Proof of Theorem~\ref{thm:3.1}]
Fix $k\ge1$, and let $Z_j=\Phi_j(A_j\times C_j)$ enumerate the chart images of Proposition~\ref{prop:3.4}.
If $\lambda(Z_j)=0$, disintegration gives
\begin{equation*}
 0=\sigma(Z_j)=\int\sigma_{\S}(Z_j)\dd\nu_0(\S),
\end{equation*}
so $\sigma_{\S}(Z_j)=0$ for almost every leaf.
For each positive-measure chart, Lemma~\ref{lem:3.5} and
\eqref{eq:3.12} give
\begin{equation*}
 \sigma_{\S}(Z_j)>0\text{ implies }
 \sigma_{\S}|_{Z_j}\sim\Hh^k|_{\S\cap Z_j}
\end{equation*}
outside a $\nu_0$-negligible family. Indeed, the restricted quotient
$\sigma_{\S}(Z_j)\dd\nu_0$ is equivalent to $d\nu_0$ on the family
where $\sigma_{\S}(Z_j)>0$.
Since the charts are countable, these conclusions hold simultaneously
outside one negligible family of leaves. Lemma~\ref{lem:3.3} and disintegration
also give $\sigma_{\S}(\ri\S)=1$ almost everywhere.
We henceforth retain only leaves in $\T^{k,0}$ satisfying all these
conclusions.

Fix a leaf outside the exceptional families. Some section $\S\cap Z_{j_0}$ has positive conditional mass. Choose a point $p$ in this section. For any $y\in\ri\S$, Proposition~\ref{prop:3.4} gives a chart section containing both $p$ and $y$. Its intersection with $\S\cap Z_{j_0}$ is nonempty and relatively open, hence has positive $\Hh^k$-measure and positive conditional mass. The local density formula therefore applies on this section. Any chart section containing $y$ overlaps it in a relatively open neighbourhood of $y$ and therefore also has positive conditional mass. Hence $\sigma_{\S}$ and $\Hh^k$ are equivalent on $\ri\S$. Since $\Hh^k(\S\setminus\ri\S)=0$, we have
\begin{equation*}
    d\sigma_{\S}=h^0_{\S}\dd\Hh^k|_\S,\quad
0<h^0_{\S}<\infty\quad\Hh^k\text{-almost everywhere}.
\end{equation*}
Define $g_{\S}=h^0_{\S}/w$ almost everywhere. On each chart section,
\begin{equation}\label{eq:3.14}
 g_{\S_a}(\Phi(a,q))
 =\frac{\sigma_{\S_a}(Z)}{W(a)J_a}f(a,q)
 \quad\text{for almost every }q\in C.
\end{equation}
Every nonempty chart section of a retained leaf has positive conditional
mass, by the overlap argument above. Hence \eqref{eq:3.12} transfers
the exceptional anchors in Lemma~\ref{lem:3.5} to one
$\nu_0$-negligible family of leaves, simultaneously for all charts.
The right-hand side of \eqref{eq:3.14} is a positive continuous
representative with locally Lipschitz logarithm and satisfies the
pointwise contraction inequality of that lemma. On an overlap,
two such representatives agree almost everywhere and therefore
everywhere by continuity. They consequently define a positive
function $g_{\S}$ on $\ri\S$ with
$\log g_{\S}\in\Lip_{\mathrm{loc}}(\ri\S)$.

The pair-covering property of Proposition~\ref{prop:3.4} and
\eqref{eq:3.10} give \eqref{eq:3.2} for $o,x\in\ri\S$.
For an arbitrary $o\in\S$, approximate $o$ by points of $\ri\S$.
If $x\in\ri\S$ and $t>0$, then $o+t(x-o)\in\ri\S$, so
continuity gives \eqref{eq:3.2} in its stated form.

All exceptional families may be chosen Borel. The local expressions
in \eqref{eq:3.14} are already jointly Borel and continuous in the
leaf coordinate. Choosing the first chart containing each point
therefore gives a jointly Borel representative on the relative
interiors of the retained leaves. No further regularisation is
needed. The normalisation \eqref{eq:3.1} follows from
$\sigma_{\S}(\S)=1$ and Lemma~\ref{lem:3.3}.

Finally, write $d\mu=(h/w)\dd\sigma$. Applying the disintegration
of $\sigma$ to the nonnegative Borel function $Hh/w$ and using
the density of $\sigma_{\S}$, we obtain
\begin{equation*}
 \int_{\R^n}H\dd\mu=\int_{\R^n}\frac{Hh}{w}\dd\sigma=\int_{\CC(\R^n)}
   \left(\int_{\S}\frac{Hh}{w}\dd\sigma_{\S}\right)
   \dd\nu_0(\S)=\int_{\CC(\R^n)}
   \left(\int_{\S}Hhg_{\S}\dd\Hh^{\dim\S}\right) \dd\nu_0(\S).
\end{equation*}
Taking $H=1$ gives
\begin{equation*}
 \int_{\CC(\R^n)}m_{\S}\dd\nu_0(\S)=\mu(\R^n)<\infty.
\end{equation*}
Lemma~\ref{lem:2.2} therefore identifies the quotient as
$d\nu=d\S_\#\mu=m_{\S}\dd\nu_0$ and gives the normalised conditionals
in \eqref{eq:3.3}, with $0<m_{\S}<\infty$ for $\nu$-almost every
leaf. Applying the same change-of-density formula to a second
positive auxiliary weight and using uniqueness of disintegration
proves the final assertion.
\end{proof}

\begin{corollary}[The geometric measure]\label{cor:3.7}
For almost every positive-dimensional leaf, the measure
\begin{equation*}
   d \lambda_{\S}=g_{\S}\dd\Hh^k|_\S
\end{equation*}
is locally finite, has support $\S$, and satisfies the homothety inequality
\begin{equation*}
    \lambda_{\S}(o+t(B-o))\ge t^n\lambda_{\S}(B)\text{ for }o\in\S,\ 0<t\le1
\end{equation*}
for every Borel $B\subset\S$. Thus it satisfies $\MCP(0,n)$ when $n>1$.
\end{corollary}
\begin{proof}
Fix $o\in\ri\S$. On each compact $C\subset\S$, the set $(o+C)/2$ is a compact subset of $\ri\S$. Theorem~\ref{thm:3.1} gives $g_{\S}(x)\le2^{n-k}g_{\S}((o+x)/2)$ for $x\in C\cap\ri\S$. Continuity therefore bounds $g_{\S}$ on $C\cap\ri\S$, proving local finiteness. Positivity of $g_{\S}$ gives full support. Multiplying \eqref{eq:3.2} by the tangential Jacobian $t^k$ and integrating proves the inequality; relative boundaries are $\Hh^k$-negligible.
\end{proof}

\subsection{First-order estimates}
\begin{corollary}[Directional bounds and sharpness]\label{cor:3.8}
Let $\S$ be a $k$-dimensional leaf for which Theorem~\ref{thm:3.1} holds. Put $d=n-k$ and, for a unit vector $v\in V_{\S}$, define
\begin{equation*}
    \ell_\pm(x,v)=\sup\{s\ge0:x\pm sv\in\S\}.
\end{equation*}
At almost every $x\in\ri\S$, simultaneously for every unit $v\in V_{\S}$,
\begin{equation}\label{eq:3.15}
-\frac{d}{\ell_+(x,v)}\le D\log g_{\S}(x)(v)\le\frac{d}{\ell_-(x,v)},
\end{equation}
where $1/\infty=0$.
In particular,
\begin{equation}\label{eq:3.16}
\norm{D_{\S}\log g_{\S}(x)}\le\frac{d}{\dist(x,\partial\S)}.
\end{equation}
On a maximal chord contained in $\ri\mathcal{S}$, with affine coordinate interval $(a,b)$, one has for $a<r<s<b$
\begin{equation}\label{eq:3.17}
\left(\frac{b-s}{b-r}\right)^d\le\frac{g_{\S}(s)}{g_{\S}(r)}
\le\left(\frac{s-a}{r-a}\right)^d,
\end{equation}
where a ratio with an infinite endpoint is interpreted as $1$. The density is constant along every complete line contained in $\ri\S$, and is constant on every leaf that is an affine subspace. The exponent and the derivative constant are sharp in every positive codimension.
\end{corollary}
\begin{proof}
At a differentiability point of $\log g_{\S}$, use \eqref{eq:3.2} to differentiate the logarithm of $g_{\S}$ at $t=1$ from the left:
\begin{equation*}
    D_{\S}\log g_{\S}(x)(x-o)\leq d\text{ for }o\in\S.
\end{equation*}
Letting $o$ approach the endpoints of the chord through $x$ in direction $v$ gives \eqref{eq:3.15}. Taking $o$ in the largest relative ball centred at $x$ gives \eqref{eq:3.16}. Apply \eqref{eq:3.2} with a centre approaching each endpoint of the chord to obtain \eqref{eq:3.17}. Passing to the limit when the chord is unbounded proves the remaining cases and the constancy assertions.

For sharpness, write $n=d+k$ with $d\ge1$, and consider
\[
u(z,y)=(\norm{z},y),\quad(z,y)\in\R^{d+1}\times\R^{k-1}.
\]
Its nontrivial leaves are $[0,\infty)\theta\times\R^{k-1}$ for $\theta\in\mathbb S^d$. Polar coordinates give $g_{\S}(r,y)=c_{\S}r^d$, $r\in [0,\infty), y\in\R^{k-1}$, with $c_{\S}>0$. Thus $D\log g_{\S}(r,y)(e_r)=\norm{D_{\S}\log g_{\S}(r,y)}=d/r$, attaining \eqref{eq:3.15} and \eqref{eq:3.16}.
\end{proof}

\section{Measure contraction and the dimension parameter}\label{sec:contraction}
We now combine the geometric estimate of Theorem~\ref{thm:3.1} with the ambient measure contraction inequality. The conditional measures satisfy the same measure contraction condition, with the same dimension parameter. Proposition~\ref{prop:4.3} shows that this parameter cannot, in general, be decreased.

\begin{theorem}[Measure contraction on isometric leaves]\label{thm:4.1}
Let $X\subset\R^n$ be closed and convex with nonempty interior, and let
\begin{equation*}
d\mu=h\dd\lambda|_ X,\quad0<h<\infty\quad\lambda\text{-almost everywhere on }X,
\end{equation*}
be finite. Suppose $(X,\norm{\cdot},\mu)$ satisfies $\MCP(\kappa,N)$, with $\kappa\in\R$, $1<N<\infty$ and $N\ge n$. Let $u:\R^n\to\R^m$ be $1$-Lipschitz, where $m\ge1$, and let $\nu$ and $(\mu_{\S})$ be the quotient measure and the conditional measures of its leaf disintegration. Then, for $\nu$-almost every $\S$,
\begin{equation*}
    \supp\mu_{\S}=\S\cap X,\quad(\S\cap X,\norm{\cdot},\mu_{\S})\text{ satisfies }\MCP(\kappa,N).
\end{equation*}
Moreover, $\mu_{\S}\sim\Hh^k|_{\S\cap X}$, where $k=\dim\S$.
\end{theorem}
\begin{proof}
Let $\Omega=\ri X$, and apply Theorem~\ref{thm:3.1} to the density $h\mathbf{1}_X$. The sets $\partial X$ and $\Omega\setminus L_h$ have Lebesgue measure zero. Disintegration of the positive auxiliary measure $d\sigma=w\dd\lambda$ in that theorem shows that their intersections with almost every leaf have the corresponding Hausdorff measure zero. The same family of leaves has full $\nu$-measure, since $\nu=\S_\#\mu\ll\nu_0$. Consequently, for $\nu$-almost every positive-dimensional leaf, $\S\cap\Omega\ne\emptyset$ and
\begin{equation*}
Y=\S\cap X,\quad I=\ri\S\cap\Omega=\ri Y,\quad
 d\mu_{\S}=h_{\S}\dd\Hh^k|_Y,\quad
h_{\S}=\frac{hg_{\S}}{\int_Y hg_{\S}\dd\Hh^k},
\end{equation*}
with $h_{\S}>0$ almost everywhere on $I$ and $\Hh^k(I\setminus L_h)=0$. The convex set $I$ is dense in $Y$, and $Y\setminus I$ has $\Hh^k$-measure zero. This proves both $\supp\mu_{\S}=Y$ and the asserted equivalence of measures.

All exceptional leaves have now been removed, independently of the centre and the contraction parameter. Fix a remaining leaf, $o\in Y$ and $0<t<1$. The homothety $H_{o,t}$ maps $I$ into $I$, and its restriction to $\Aff\S$ is nonsingular. Therefore both $x$ and $H_{o,t}(x)$ belong to $L_h$ for $\Hh^k$-almost every $x\in I$. Theorem~\ref{thm:3.1} applies at every point of $I$, while Lemma~\ref{lem:2.1} applies whenever both endpoints belong to $L_h$. Together they give
\begin{equation*}
h_{\S}(H_{o,t}(x))\ge t^{n-k}\frac{h(H_{o,t}(x))}{h(x)}h_{\S}(x)
\ge t^{-k}\beta_{\kappa,N}^{(t)}(\norm{x-o})h_{\S}(x)
\end{equation*}
for almost every $x\in I\cap B(o,R_{\kappa,N})$. Since the $k$-dimensional Jacobian of $H_{o,t}$ is $t^k$, integration yields \eqref{eq:2.5} for every Borel set in $ I\cap B(o,R_{\kappa,N})$. Both the relative boundary and its homothetic image are $\Hh^k$-negligible, so the assertion extends to sets in $Y\cap B(o,R_{\kappa,N})$. The case $t=1$ is immediate. A zero-dimensional conditional is a Dirac mass and satisfies the same assertion.
\end{proof}

\begin{proposition}[Optimality of the inherited dimension]\label{prop:4.3}
For every $2\le n\le N<\infty$ there is a weighted $\CD(0,N)$ measure on a closed Euclidean ball, with positive smooth density on its interior, and a scalar $1$-Lipschitz map whose nontrivial conditional measures fail $\MCP(0,N')$ for every $1<N'<N$.
\end{proposition}
\begin{proof}
Let $X=\overline{B(0,1)}$ and $p=-e_1$. Choose $c>0$ so that the density
\[
h(x)=\begin{cases}
c(1+x_1)^{N-n},&N>n,\\
c,&N=n.
\end{cases}
\]
has integral one on $X$. It is positive and smooth in $\ri X$. For $N>n$, $\rho(x)=-\log c-(N-n)\log(1+x_1)$ satisfies
\begin{equation*}
D^2\rho-\frac{D\rho\otimes D\rho}{N-n}=0.
\end{equation*}
By the characterisation recalled in Section~\ref{sec:curvature}, $h\dd\lambda|_X$ satisfies $\CD(0,N)$; the case $N=n$ is the unweighted ball. Alternatively, the affine inequality $1+(H_{o,t}(x))_1\ge t(1+x_1)$ verifies $\MCP(0,N)$ directly.

The leaves of $u(\cdot)=\norm{\cdot-p}$ are the closed rays issuing from $p$. For $\theta\in\mathbb S^{n-1}$ with $\theta_1>0$, write $x=p+r\theta$, $0<r<2\theta_1$. Polar coordinates give
\begin{equation*}
h(x)\dd\lambda(x)=c\,a(\theta)r^{N-1}\dd r\dd\Hh^{n-1}(\theta),\qquad
 a(\theta)=\begin{cases}\theta_1^{N-n},&N>n,\\1,&N=n.\end{cases}
\end{equation*}
Hence the conditional on the ray is
\begin{equation*}
d\mu_\theta(r)=\frac{N}{(2\theta_1)^N}r^{N-1}\dd r.
\end{equation*}
Contraction towards the endpoint $r=0$ multiplies the mass of every source set by exactly $t^N$. Since $t^N<t^{N'}$ for $0<t<1$ and $N'<N$, \eqref{eq:2.5} cannot hold with $N'$. The failure persists if the centre is required to lie in the relative interior. Indeed, for a source near $r>0$ and centre $\epsilon>0$, the necessary density inequality is
\begin{equation*}
t\bigl(\epsilon+t(r-\epsilon)\bigr)^{N-1}\ge t^{N'}r^{N-1},
\end{equation*}
which fails for all sufficiently small $\epsilon$.
\end{proof}

\section{Monotone fibres and convex faces}\label{sec:monotone}
In this section we prove that the measure-contraction property is inherited by the conditional measures associated with the fibres of a maximal monotone relation. We first realise these fibres as translates of the leaves of its resolvent. We then prove convergence of the corresponding conditional measures in total variation as the resolvent parameter tends to zero. The result applies, in particular, to subdifferentials of extended-valued convex functions.

\subsection{Maximal extensions and resolvents}
We identify a relation $A:\R^n\rightrightarrows\R^n$ with its graph. It is monotone if
\begin{equation*}
    \ip{x-y}{p-q}\ge0\text{ for all }(x,p),(y,q)\in A.
\end{equation*}
A monotone relation is maximal if its graph is not properly contained in the graph of another monotone relation. Let us recall that the graph of a maximal monotone relation is closed, its values and inverse values are closed and convex, and its resolvents are everywhere defined. We refer to \cite{r1,r4} for these facts and for Minty's theorem.

The domain of a relation $A:\R^n\rightrightarrows\R^n$ is
\begin{equation*}
 \dom A=\{x\in\R^n\mid A(x)\ne\emptyset\}.
\end{equation*}
Thus $A$ may have empty values outside its domain, even though
its ambient source space is $\R^n$.

The following lemma collects standard facts about monotone operators
and Borel selections (see \cite[Corollary~1.5 and Theorem~2.2]{r1}
and \cite[Theorem~18.18]{r23}); we include a proof for completeness.

\begin{lemma}\label{lem:5.1}
A monotone map $T:\R^n\to\R^n$ defined on all of $\R^n$ has a unique maximal monotone extension. For any maximal monotone relation $A$, the multivalued locus
\begin{equation*}
M_A=\{x\mid\#A(x)>1\}
\end{equation*}
is Borel and $\lambda$-negligible. The domain of $A$ is Borel and admits a Borel selection. Consequently, all selections agree $\lambda$-almost everywhere on the domain.
\end{lemma}
\begin{proof}
Existence of a maximal extension follows from the Kuratowski--Zorn lemma. Let $A_1,A_2$ be two such extensions, take $(x,p)\in A_1$, $(y,q)\in A_2$, and set $z_t=(1-t)x+ty$, where $0<t<1$. Monotonicity with $(z_t,T(z_t))$ gives
\begin{equation*}
\ip{p-T(z_t)}{x-y}\ge0,\quad\ip{T(z_t)-q}{x-y}\ge0.
\end{equation*}
Their sum shows that $A_1\cup A_2$ is monotone. Maximality implies $A_1=A_2$.

Now let $A$ be any maximal monotone relation. The projection of its graph restricted to a compact ball is compact; hence its domain $\dom A$ is a countable union of compact sets. The same argument, applied to
\begin{equation*}
    \{(x,p,q)\in(\R^n)^3\mid(x,p),(x,q)\in A,\ \norm{p-q}\ge 1/j,\ \norm{x},\norm{p},\norm{q}\le m\},
\end{equation*}
shows that $M_A$ is Borel. Replacing $\norm{p-q}$ by $|\ip{p-q}{e}|$ proves the same assertion for the locus on which the projection of $A(x)$ onto a fixed coordinate direction $e$ is multivalued. Fix such a direction and write $x=y+se$, with $y\perp e$. The sets $A(x)$, $x\in \R^n$, for a maximal monotone relation $A$ are convex, and therefore the sets
\begin{equation*}
I_{y,s}=\{\ip{p}{e}\mid p\in A(y+se)\}
\end{equation*}
are intervals, possibly empty or unbounded. Monotonicity gives $\sup I_{y,s}\le\inf I_{y,t}$ whenever $s<t$ and both intervals are nonempty. For fixed $y$, the nondegenerate intervals therefore have pairwise disjoint interiors and are thus countable. Fubini's theorem shows that the projection of $A(x)$ onto $\R e$ is a singleton for almost every $x$ in the domain. The finitely many coordinate directions give $\lambda(M_A)=0$.

Finally, the sections of the closed graph over its domain are nonempty and closed, hence $\sigma$-compact. The Borel uniformisation theorem \cite[Theorem~18.18]{r23} supplies a Borel selection. One may also take the unique element of smallest norm in each nonempty closed convex value. Since $\lambda(M_A)=0$, all selections agree almost everywhere.
\end{proof}

For $\epsilon>0$, we denote the resolvent and the Yosida approximation by
\begin{equation*}
J_\epsilon=(\Id+\epsilon A)^{-1},\quad
A_\epsilon=\epsilon^{-1}(\Id-J_\epsilon).  
\end{equation*}

The resolvent $J_\epsilon$ is everywhere defined and firmly nonexpansive, see \cite[Corollary~23.9]{r4}. 
In particular, $J_\epsilon$ is $1$-Lipschitz by the
Cauchy--Schwarz inequality.

\begin{lemma}\label{lem:5.2}
Let $A:\R^n\rightrightarrows\R^n$ be maximal monotone and let
$\epsilon>0$. The isometric leaves of $J_\epsilon$ are precisely
\begin{equation*}
 \S_{\epsilon,p}=\epsilon p+A^{-1}(p)
 \text{ for }p\in\ran A.
\end{equation*}
These are exactly the nonempty fibres of the Lipschitz map
$A_\epsilon=\epsilon^{-1}(\Id-J_\epsilon)$.
\end{lemma}
\begin{proof}
For $z,w\in\R^n$, put
$x=J_\epsilon z$, $y=J_\epsilon w$,
$p=A_\epsilon z$ and $q=A_\epsilon w$.
Then $p\in A(x)$, $q\in A(y)$, and monotonicity gives
\begin{equation*}
\norm{z-w}^2
=\norm{x-y}^2+2\epsilon\ip{x-y}{p-q}
   +\epsilon^2\norm{p-q}^2
\ge\norm{x-y}^2+\epsilon^2\norm{p-q}^2.
\end{equation*}
Hence $J_\epsilon$ preserves the distance between $z$ and $w$
if and only if $A_\epsilon z=A_\epsilon w$.
Its leaves are therefore precisely the nonempty fibres of
$A_\epsilon$. Finally, the resolvent identity gives
\begin{equation*}
A_\epsilon z=p
\text{ if and only if }z-\epsilon p\in A^{-1}(p),
\end{equation*}
which proves the assertion.
\end{proof}

\subsection{Total variation under disintegration}
For a finite signed measure $\sigma$ on $Z$, we write $\norm{\sigma}_{\TV}=|\sigma|(Z)$. Conditional probability kernels are chosen as everywhere-defined Borel versions, using a fixed probability measure on the exceptional Borel sets where disintegration does not determine them.

\begin{lemma}[Total variation and conditional measures]\label{lem:5.3}
Let $Z,P$ be Polish spaces, let $q:Z\to P$ be Borel, and let $\xi_0,\xi_1$ be finite nonnegative measures. Write for $i=0,1$
\begin{equation*}
    \xi_i=\int_P\nu_{i,p}\dd\theta_i(p),\qquad\theta_i=q_\#\xi_i.
\end{equation*}
If $\vartheta$ dominates both quotient measures and $a_i=d\theta_i/d\vartheta$, $i=0,1$, then
\begin{equation}\label{eq:5.1}
\norm{\xi_1-\xi_0}_{\TV}=\int_P\norm{a_1(p)\nu_{1,p}-a_0(p)\nu_{0,p}}_{\TV}\dd\vartheta(p).
\end{equation}
In particular,
\begin{equation}\label{eq:5.2}
\int_P\norm{\nu_{1,p}-\nu_{0,p}}_{\TV}\dd\theta_0(p)
\le\norm{\xi_1-\xi_0}_{\TV}+\norm{\theta_1-\theta_0}_{\TV}
\le2\norm{\xi_1-\xi_0}_{\TV}.
\end{equation}
\end{lemma}
\begin{proof}
Put $\xi=\xi_0+\xi_1$ and $f_i=d\xi_i/d\xi$, and disintegrate $\xi=\int\nu_p\dd\theta(p)$, where $\theta=q_\#\xi$. Uniqueness of disintegration gives
\begin{equation*}
    a_i(p)\dd\nu_{i,p}=\frac{d\theta}{d\vartheta}(p)\,f_i\dd\nu_p
\quad\text{for }\vartheta\text{-almost every }p.
\end{equation*}
The total variation of the difference on the right is obtained by integrating $|f_1-f_0|$ against $\frac{d\theta}{d\vartheta}(p)\nu_p$. Integration in $p$ proves \eqref{eq:5.1}. All integrands are measurable: the variation of a signed kernel is the supremum of the sums of absolute masses over finite partitions from a countable generating algebra. These are also standard facts about measure kernels; see \cite{r8}.

The pointwise triangle inequality gives
\begin{equation*}
    a_0\norm{\nu_{1,p}-\nu_{0,p}}_{\TV}
\le\norm{a_1\nu_{1,p}-a_0\nu_{0,p}}_{\TV}+|a_1-a_0|.
\end{equation*}
Integrating against $\vartheta$ and using \eqref{eq:5.1} proves the first inequality in \eqref{eq:5.2}. The second follows because pushforward decreases total variation.
\end{proof}

\begin{proposition}[Conditional convergence of the resolvents]\label{prop:5.4}
Let $A\colon\R^n\rightrightarrows\R^n$ be maximal monotone and let $d\mu=h\dd\lambda$ be finite, with $\mu(\R^n\setminus\dom A)=0$. Choose a Borel selection $T$ of $A$ on its domain and set $T=0$ off the domain. Write
\begin{equation*}
    \mu=\int\mu_p\dd\eta(p)=\int\mu_p^\epsilon\dd\eta_\epsilon(p)\text{, where }
\eta=T_\#\mu\text{ and }\eta_\epsilon=(A_\epsilon)_\#\mu.
\end{equation*}
Set $t^{\epsilon}_p(x)=x-\epsilon p$ for $x\in \R^n$ and
\begin{equation*}   
\widehat\mu_p^\epsilon={t^{\epsilon}_p}_\#\mu_p^\epsilon.
\end{equation*}
Then
\begin{equation}\label{eq:5.3}
\lim_{\epsilon\to 0^+}\int\norm{\widehat\mu_p^\epsilon-\mu_p}_{\TV}\dd\eta(p)=0.
\end{equation}
There is a sequence $(\epsilon_j)_{j=1}^{\infty}$ convergent to $0$ along which the convergence holds in total variation for $\eta$-almost every $p$.
\end{proposition}
\begin{proof}
\textbf{Step 1: graph densities.}

Let $G=\graph A\subset \R^n\times \R^n$. Let  $\pi_x,\pi_p\colon \R^n\times\R^n\to\R^n$ be the coordinate projections. Define $\Phi_\epsilon(x,p)=x+\epsilon p$ for $(x,p)\in G$. Monotonicity gives
\begin{equation}\label{eq:5.4}
\norm{(x-y)+\epsilon(p-q)}^2\ge\norm{x-y}^2+\epsilon^2\norm{p-q}^2\text{ for }p\in A(x),q\in A(y).
\end{equation}
By Minty's theorem, $\Phi_\epsilon$ is therefore a bi-Lipschitz bijection from $G$ onto $\R^n$. In particular, $G$ is $n$-rectifiable and has finite $\Hh^n$-measure on bounded subsets. Put
\begin{equation*}
   \gamma_\epsilon=(\Phi_\epsilon^{-1})_\#\mu\text{ and let }
\gamma_0=R_\#\mu\text{ where }R(x)=(x,T(x))\text{ for }x\in \R^n.
\end{equation*}
Since $G$ is countably $n$-rectifiable and $\Hh^n|_G$ is locally
finite, Theorem~3.2.19 of \cite{r21}, applied locally, gives a
unique approximate tangent plane $ P_\xi\subset \R^{2n}$
for $\Hh^n$-almost every $\xi\in G$.
For either of the linear maps $F=\Phi_\epsilon$ and $F=\pi_x$,
define
\begin{equation*}
 D_GF(\xi)=DF|_{P_\xi},
 \text{ and }
 J_GF(\xi)
 =\sqrt{\det
   \bigl((D_GF(\xi))^*D_GF(\xi)\bigr)}.
\end{equation*}
Thus $J_GF$ measures the
$n$-dimensional volume distortion of $F$ along the graph.
In particular,
\begin{equation*}
D_G\Phi_\epsilon(\xi)=(\pi_x+\epsilon\pi_p)|_{P_\xi}\text{ and }D_G\pi_x(\xi)=\pi_x|_{P_\xi}.
\end{equation*}
Set these Jacobians equal to zero on the negligible set where
the approximate tangent plane is undefined.

Since $\Phi_\epsilon:G\to\R^n$ is a bi-Lipschitz bijection,
the area formula \cite[Section~3.2.22]{r21} gives, for every
Borel set $E\subset G$,
\begin{equation*}
 \gamma_\epsilon(E)=\mu(\Phi_\epsilon(E))=\int_{\Phi_\epsilon(E)}h(z)\dd\lambda(z)=\int_E h(\Phi_\epsilon(\xi))
          J_G\Phi_\epsilon(\xi)\dd\Hh^n(\xi).
\end{equation*}
For the projection $\pi_x$, first consider
$B=G\cap\pi_x^{-1}(M_A)$.
Lemma~\ref{lem:5.1} gives $\lambda(M_A)=0$, and the area formula
yields
\begin{equation*}
 \int_B J_G\pi_x\dd\Hh^n
 =\int_{M_A}\Hh^0(A(x))\dd\lambda(x)=0.
\end{equation*}
Hence $J_G\pi_x=0$ at $\Hh^n$-almost every
point of $B$.
On $G\setminus B$, the projection $\pi_x$ is injective, with
inverse $x\mapsto(x,T(x))$ on $\dom A\setminus M_A$.
Applying the weighted area formula on this part of the graph
therefore gives for Borel $E\subset G$
\begin{equation*}
\int_E h(\pi_x(\xi))J_G\pi_x(\xi)\dd\Hh^n(\xi)=\int_{\dom A\setminus M_A}
      \mathbf{1}_E(x,T(x))h(x)\dd\lambda(x)=\gamma_0(E).
\end{equation*}
The last equality uses $\mu(M_A)=0$ and
$\mu(\R^n\setminus\dom A)=0$.
We may take $h$ to be a finite nonnegative Borel representative,
so its product with the vanishing Jacobian on $B$ is unambiguous.
Consequently, as measures on $G$,
\begin{equation}\label{eq:5.5}
 d\gamma_\epsilon=(h\circ\Phi_\epsilon)J_G\Phi_\epsilon\dd\Hh^n|_G\text{ and }
 d\gamma_0=(h\circ\pi_x)J_G\pi_x\dd\Hh^n|_G.
\end{equation}

\textbf{Step 2: total-variation convergence.}

We claim that
\begin{equation}\label{eq:5.6}
\lim_{\epsilon\to 0^+}\norm{\gamma_\epsilon-\gamma_0}_{\TV}=0.
\end{equation}
Let $E\subset G$ be a bounded Borel set. On $E$, $\Phi_\epsilon$ converges uniformly to $\pi_x$. Their tangential differentials are the restrictions of the linear maps $\pi_x+\epsilon\pi_p$ to the fixed approximate tangent plane of $G$. Hence the Jacobians converge almost everywhere and are uniformly bounded for $0<\epsilon\le1$. Given $\delta>0$, choose $g\in \mathcal{C}_c(\R^n)$ with $\norm{h-g}_{L^1(\lambda)}<\delta$. The map $\Phi_\epsilon$ is injective and $\pi_x$ has multiplicity one almost everywhere on its image. The area formula therefore gives
\begin{equation*}
  \int_E|h-g|(\Phi_\epsilon)J_G\Phi_\epsilon\dd\Hh^n\le\delta,\qquad
\int_E|h-g|(\pi_x)J_G\pi_x\dd\Hh^n\le\delta.  
\end{equation*}
Thus
\begin{equation}\label{eq:5.7}
\int_E\bigl|h(\Phi_\epsilon)J_G\Phi_\epsilon-h(\pi_x)J_G\pi_x\bigr|\dd\Hh^n\leq 2\delta+\int_E\bigl|g(\Phi_\epsilon)J_G\Phi_\epsilon-g(\pi_x)J_G\pi_x\bigr|\dd\Hh^n.
\end{equation}
The graph measures $\gamma_\epsilon, \gamma_0$ have total mass $\mu(\R^n)$. Hence, for every bounded Borel $E\subset G$,
\begin{equation*}
   \norm{\gamma_\epsilon-\gamma_0}_{\TV}
\le2\norm{(\gamma_\epsilon-\gamma_0)|_E}_{\TV}+2\gamma_0(G\setminus E). 
\end{equation*}
Using \eqref{eq:5.7}, let $\epsilon$ tend to zero, then let $\delta$ tend to zero and, finally, exhaust $G$ by bounded sets. This proves \eqref{eq:5.6}.

\textbf{Step 3: convergence of conditional measures.}

Regard $\gamma_\epsilon$ and $\gamma_0$ as measures on $\R^{2n}$.
Using the prescribed everywhere-defined Borel versions of the
conditional probabilities, define
\begin{equation*}
 \gamma_{\epsilon,p}
   =(t^{\epsilon}_p,p)_\#\mu_p^\epsilon\text{ and }
 \gamma_{0,p}=(\mathrm{id},p)_\#\mu_p.
\end{equation*}
For $\eta_\epsilon$-almost every $p$, the measure $\mu_p^\epsilon$
is concentrated on $A_\epsilon^{-1}(p)$.
Likewise, $\mu_p$ is concentrated on $T^{-1}(p)$, for
$\eta$-almost every $p$. Thus these kernels give the disintegrations
\begin{equation*}
 \gamma_\epsilon=\int\gamma_{\epsilon,p}\dd\eta_\epsilon(p),
 \text{ and }
 \gamma_0=\int\gamma_{0,p}\dd\eta(p)
\end{equation*}
with respect to $\pi_p$.
Their definitions also give
\begin{equation*}
 (\pi_x)_\#\gamma_{\epsilon,p}=\widehat\mu_p^\epsilon
 \text{ and }
 (\pi_x)_\#\gamma_{0,p}=\mu_p
 \text{ for every }p.
\end{equation*}
The definitions at exceptional labels ensure these identities
everywhere, although concentration on $G$ is required only almost
everywhere for the respective quotient measures.
Contraction of total variation under $\pi_x$ and
Lemma~\ref{lem:5.3} therefore give
\begin{equation}\label{eq:5.8}
 \int\norm{\widehat\mu_p^\epsilon-\mu_p}_{\TV}\dd\eta(p)\le\int\norm{\gamma_{\epsilon,p}-\gamma_{0,p}}_{\TV}\dd\eta(p)\le
 \norm{\gamma_\epsilon-\gamma_0}_{\TV}
 +\norm{\eta_\epsilon-\eta}_{\TV}
\le2\norm{\gamma_\epsilon-\gamma_0}_{\TV}.
\end{equation}
This proves \eqref{eq:5.3}. Choose a sequence $(\epsilon_j)_{j=1}^{\infty}$ so that the right-hand sides  with parameters $\epsilon_j$ are summable; Tonelli's theorem gives the almost-everywhere assertion.
\end{proof}

\subsection{Measure contraction property and the conditional measure class}

Proposition~\ref{prop:5.4} gives convergence of the translated
resolvent conditionals. To pass their measure contraction properties to
the limit, we use the following stability statement. It allows
both the measures and the centres to vary and retains the
source-radius restriction when $\kappa>0$.

\begin{lemma}[Stability of homothety domination]\label{lem:4.2}
Let $(\eta_j)_{j=1}^{\infty}$ be finite measures on $\R^n$ converging narrowly to a
finite measure $\eta$ concentrated on a closed convex set $Y$.
Let $G\subset Y$ and $T\subset(0,1)$ be dense. For each $o\in G$,
suppose there are $(o_j)_{j=1}^{\infty}$ converging to $o$ such that, for all sufficiently large
$j$, \eqref{eq:2.6} holds for $\eta_j$, centre $o_j$, and every
$t\in T$. Then \eqref{eq:2.6} holds for $\eta$ at every
$o\in Y$ and $0<t\le1$. The same assertion holds for a fixed
measure if domination is initially known only on $G\times T$.
\end{lemma}
\begin{proof}
Write $R=R_{\kappa,N}\in(0,\infty]$. Fix $o\in G$, $t\in T$,
a nonnegative bounded continuous function $\phi$, and
$0\le\chi\le1$ with $\chi\in \mathcal{C}_c(B(o,R))$.
For all sufficiently large $j$, the support of $\chi$ lies in
$B(o_j,R)$, so \eqref{eq:2.6} gives
\begin{equation*}
    \int\chi(x)\phi(H_{o_j,t}(x))
       \beta_{\kappa,N}^{(t)}(\norm{x-o_j})\dd\eta_j(x)
 \le\int\phi\dd\eta_j.
\end{equation*}
The integrands on the left, extended by zero off the support of
$\chi$, converge uniformly to the corresponding integrand with
centre $o$. Narrow convergence yields the same inequality for
$\eta$ and $o$. Choose increasing such cutoffs $\chi$ converging to
$\mathbf{1}_{B(o,R)}$ and apply monotone convergence. This proves
\eqref{eq:2.6} on $G\times T$.

For $(o,t)\in Y\times(0,1)$, take
$(o_\ell,t_\ell)_{\ell=1}^{\infty}$ in $G\times T$ converging to $(o,t)$ and repeat
the same compactly supported cutoff argument with the fixed measure
$\eta$. Continuity of the distortion coefficient in the centre and
time gives the desired domination. At $t=1$ it is immediate.
Regularity of finite Borel measures allows passage from continuous
test functions to the measure inequality.
\end{proof}

We now combine the inheritance of measure contraction property for leaves with conditional
convergence and stability to obtain measure contraction property on
 fibres of monotone maps.

\begin{theorem}[Measure contraction property on fibres of monotone maps]\label{thm:5.5}
Let $X,\mu,\kappa,N$ satisfy the assumptions of Theorem~\ref{thm:4.1}. Let $A:\R^n\rightrightarrows\R^n$ be maximal monotone and suppose $\mu(\R^n\setminus\dom A)=0$. Disintegrate $\mu$ by any Borel selection $T$ of $A$ on its domain, extended by zero elsewhere:
\begin{equation*}
\mu=\int\mu_p\dd\eta(p),\quad\eta=T_\#\mu.
\end{equation*}
Put $F_p=A^{-1}(p)$ and $Y_p=F_p\cap X$. For $\eta$-almost every $p$,
\begin{equation}\label{eq:5.9}
\supp\mu_p=Y_p= \cl (X\cap \dom A\cap T^{-1}(p)).
\end{equation}
Moreover, $(Y_p,\norm{\cdot},\mu_p)$ satisfies $\MCP(\kappa,N)$. If $k=\dim F_p$, then
\begin{equation}\label{eq:5.10}
\mu_p\sim\Hh^k|_{Y_p},\qquad\mu_p(\ri F_p\cap\ri X)=1.
\end{equation}
The quotient measure and, up to a quotient-negligible set, the conditional measures are independent of the selection.
\end{theorem}
\begin{proof}
Selection independence follows from Lemma~\ref{lem:5.1} and $\mu\ll\lambda$. For $\epsilon>0$, apply Theorem~\ref{thm:4.1} to a firmly nonexpansive, hence $1$-Lipschitz map $J_\epsilon$. Lemma~\ref{lem:5.2} shows that, for $\eta_\epsilon$-almost every $p$, the translated conditional $\widehat\mu_p^\epsilon$ has full support on
\begin{equation*}
Y_{\epsilon,p}=F_p\cap(X-\epsilon p)
\end{equation*}
and satisfies $\MCP(\kappa,N)$.

The exceptional quotient sets depend on $\epsilon$. To handle this, choose the sequence $\epsilon_j$ tending to zero in the proof of Proposition~\ref{prop:5.4}, so that $\sum_j\norm{\gamma_{\epsilon_j}-\gamma_0}_{\TV}<\infty$. Let $B_j$ be a Borel $\eta_{\epsilon_j}$-negligible set containing the labels where this conclusion fails. Pushforward decreases total variation, so
\begin{equation*}
\eta(B_j)\le\norm{\eta-\eta_{\epsilon_j}}_{\TV}\le\norm{\gamma_0-\gamma_{\epsilon_j}}_{\TV}.
\end{equation*}

The Borel--Cantelli lemma and Proposition~\ref{prop:5.4} therefore imply that, for $\eta$-almost every $p$, one has $p\notin B_j$ for all sufficiently large $j$ and $(\widehat\mu_p^{\epsilon_j})_{j=1}^{\infty}$ converges to $\mu_p$ in total variation. By disintegration, we may also assume that $\mu_p(\ri X)=1$ and that $\mu_p$ is concentrated on $T^{-1}(p)\cap\dom A$.

Fix such a $p$ and put $G_p=F_p\cap\ri X$.
This set is nonempty, and convexity gives $\cl G_p=Y_p$.
For each fixed $o\in G_p$, choose $r_o>0$ with
$B(o,r_o)\subset X$. Since $\epsilon_j$ tends to zero, we have
\begin{equation*}
 o+\epsilon_jp\in B(o,r_o)\subset X
 \text{ for all sufficiently large }j.
\end{equation*}
Together with $o\in F_p$, this implies
$o\in Y_{\epsilon_j,p}=F_p\cap(X-\epsilon_jp)$.
We may therefore choose the centres in the stability argument
to be $o_j=o$ for all sufficiently large $j$.
The required index may depend on $p$ and $o$, which is harmless
since the limit is taken separately for each fixed centre.
Lemma~\ref{lem:4.2} first gives the homothety inequalities for
$\mu_p$ at centres in $G_p$ and then, by density, at every centre
in $Y_p$, retaining the source-radius restriction when $\kappa>0$.

Let us show that these inequalities imply $\supp\mu_p=Y_p$. Write $R_{\kappa,N}=\pi\sqrt{(N-1)/\kappa}$ when $\kappa>0$ and $R_{\kappa,N}=\infty$ otherwise. If $o\in Y_p$ lies at distance less than $R_{\kappa,N}$ from a point of $\supp\mu_p$, choose a bounded positive-measure set $E\subset B(o,R_{\kappa,N})$ with closure contained in this ball. Contracting $E$ towards $o$, with a sufficiently small positive ratio, puts its image in any prescribed neighbourhood of $o$. The distortion coefficient has a positive lower bound on $E$, so that neighbourhood has positive measure. Thus $o\in\supp\mu_p$. If $R_{\kappa,N}$ is finite, divide the segment from a support point to any point of $Y_p$ into finitely many subsegments of length less than $R_{\kappa,N}$ and apply this implication successively. Hence $\supp\mu_p=Y_p$. Since $\mu_p$ is concentrated on $T^{-1}(p)\cap X\cap\dom A\subset Y_p$, the closure identity \eqref{eq:5.9} follows.

The equivalence of measures follows from Lemma~\ref{lem:5.6} below. The intersection $F_p\cap\ri X$ is nonempty and open relative to $F_p$, so $\dim Y_p=\dim F_p$. The relative boundary of $F_p$ is $\Hh^k$-negligible, while $\mu_p(\partial X)=0$. This proves the final assertion in \eqref{eq:5.10}.
\end{proof}
Euclidean measure rigidity under a qualitative contraction hypothesis was proved by Cavalletti and Mondino \cite[Theorem~1.1]{r12}. We include an elementary argument giving the measure equivalence needed here.

\begin{lemma}[Measure class on a convex support]\label{lem:5.6}
Let $Y$ be a nonempty closed convex subset of a finite-dimensional Euclidean space, and let $k=\dim Y$. If a nonzero finite measure $\sigma$ has support $Y$ and satisfies $\MCP(\kappa,N)$, then
\[
\sigma\sim\Hh^k|_Y.
\]
\end{lemma}
\begin{proof}
The assertion is immediate for $k=0$. Identify $\Aff Y$ with
$\R^k$, and suppose $k\ge1$. Choose $B=B(c,r)$ compactly contained in $\ri Y$,
with $2r<R_{\kappa,N}$ when $\kappa>0$. Positivity and continuity
of the distortion coefficient give $\gamma>0$ such that
\begin{equation}\label{eq:5.11}
 \sigma\bigl((o+E)/2\bigr)\ge\gamma\sigma(E)
 \text{ for all }o\in B\text{ and all Borel sets }\ E\subset B.
\end{equation}
If $\lambda^k(E)=0$, integration in $o$ and Fubini's theorem give
\begin{equation*}
 \gamma\lambda^k(B)\sigma(E)
 \le\int_B\sigma\bigl((o+E)/2\bigr)\dd \lambda^k(o)
 =\int_{Y}\lambda^k\bigl((2z-E)\cap B\bigr)\dd\sigma(z)=0.
\end{equation*}
Thus $\sigma|_B\ll\lambda^k|_B$.

Conversely, put $B'=B(c,r/3)$ and let $E\subset B'$ be Borel
with $\sigma(E)=0$. Applying \eqref{eq:5.11} to the source set
$\{x\in B'\mid(o+x)/2\in E\}$ yields
\begin{equation*}
 \sigma(\{x\in B'\mid(o+x)/2\in E\})=0
 \text{ for all }o\in B.
\end{equation*}
Integrating in $o$ and using $2E-x\subset B$ for $x\in B'$, we get
\begin{equation*}
 0=\int_{B'}\lambda^k\bigl((2E-x)\cap B\bigr)\dd\sigma(x)
  =2^k\lambda^k(E)\sigma(B').
\end{equation*}
Full support gives $\sigma(B')>0$, so $\lambda^k(E)=0$.
A countable cover by such smaller balls proves equivalence on
$\ri Y$.

The relative boundary is $\lambda^k$-negligible. Cover it by countably
many bounded Borel subsets $(E_j)_{j=1}^{\infty}$ of $\partial Y$, choosing $o_j\in\ri Y$
so that $\sup\{\norm{x-o_j}\mid x\in E_j\}<R_{\kappa,N}$ when
$\kappa>0$. Such a cover exists because $\ri Y$ is dense in $Y$.
The set $(o_j+E_j)/2$ lies in $\ri Y$ and is Lebesgue-negligible,
hence $\sigma$-negligible. The contraction inequality, whose coefficient
has a positive lower bound on $E_j$, implies $\sigma(E_j)=0$.
Thus $\sigma(\partial Y)=0$, completing the proof.
\end{proof}

For an everywhere-defined Borel monotone map, Theorem~\ref{thm:5.5} applies to the unique extension in Lemma~\ref{lem:5.1}. The closure in \eqref{eq:5.9} is necessary. For example, the monotone map on $\R$ taking the values $0$, $1/2$, and $1$ on $(-\infty,0)$, $\{0\}$, and $(0,\infty)$ has, for Lebesgue measure on $[-1,1]$, a conditional supported on $[-1,0]$ although its zero fibre in this interval is $[-1,0)$. Its maximal extension has $A(0)=[0,1]$.

The measure-contraction conclusion cannot in general be strengthened
to curvature--dimension inheritance: Corollary~\ref{cor:monotone-cd-failure}
gives an ambient $\CD(0,3)$ measure whose conditionals on the fibres
of a firmly nonexpansive map fail even $\CD(0,\infty)$.

\subsection{Faces of convex functions and their measure class}
Let $f:\R^n\to(-\infty,\infty]$ be proper, lower-semicontinuous, and convex. Its subdifferential is maximal monotone \cite{r31}; its inverse values are the exposed affine contact sets
\begin{equation*}
F_p=(\partial f)^{-1}(p)=\{x\in \R^n\mid f(x)+f^*(p)=\ip{p}{x}\},
\end{equation*}
where $f^*(p)=\sup\{\ip{p}{q}-f(q)\mid q\in \R^n\}$, $p\in\R^n$, is the Legendre transform of $f$.

Whenever $f$ is finite on an open set $\Omega$, we use a Borel
extension of $Df$ from its differentiability set in $\Omega$
to $\R^n$. Its values elsewhere do not affect disintegrations
of absolutely continuous measures concentrated on $\Omega$.

\begin{corollary}[Measure contraction property on convex-gradient fibres]\label{cor:5.7}
Under the ambient assumptions of Theorem~\ref{thm:4.1}, suppose that $f$ is finite on $\Omega=\ri X$. Disintegrate $\mu$ by $Df$ on its differentiability set in $\Omega$. For almost every label $p$, with $k=\dim F_p$,
\begin{equation*}
\supp\mu_p=F_p\cap X,\quad\mu_p\sim\Hh^k|_{F_p\cap X},\quad
\mu_p(\ri F_p\cap\Omega)=1,
\end{equation*}
and $(F_p\cap X,\norm{\cdot},\mu_p)$ satisfies $\MCP(\kappa,N)$.
\end{corollary}
\begin{proof}
Since $\Omega\subset\ri\dom f$, the subdifferential is nonempty there, and $f$ is locally Lipschitz and differentiable almost everywhere; see \cite{r30}. Thus $\mu(\dom\partial f)=\mu(X)$. A Borel selection of $\partial f$ agrees with $Df$ at every differentiability point. Apply Theorem~\ref{thm:5.5} to $A=\partial f$.
\end{proof}
By restriction and weighting, the preceding result also identifies the conditional measure class without any curvature assumption on the ambient measure.

\begin{corollary}[Conditional measure class on convex-gradient fibres]\label{cor:5.8}
Let $\Omega\subset\R^n$ be open and convex, and let $f$ be proper, lower-semicontinuous, convex, and finite on $\Omega$. Let $\mu=h\lambda|_{\Omega}$ be finite, where $h\ge0$ is Borel, and write $\mu=\int\mu_p\dd\eta(p)$ for its disintegration with respect to $Df$. For $\eta$-almost every $p$, with $k=\dim F_p$, one has
\begin{equation}\label{eq:5.12}
\mu_p\sim\Hh^k|_{F_p\cap\Omega\cap h^{-1}((0,\infty))},\qquad
\mu_p(\ri F_p\cap\Omega)=1.
\end{equation}
If $X=\cl\Omega$ and $h>0$ almost everywhere on $\Omega$, then $\supp\mu_p=F_p\cap X$ for almost every $p$.
\end{corollary}
\begin{proof}
The assertion is void if $\mu=0$. Put $X=\cl\Omega$ and
\begin{equation*}
   \psi(x)=\sqrt{1+\norm{x}^2},\quad
 dm=e^{-\psi}\dd\lambda|_X.  
\end{equation*}
This is a finite measure with positive density and support $X$.
We first verify $\MCP(-n,n+1)$ directly. Since $\psi$ is
$1$-Lipschitz, for $o,x\in X$, $r=\norm{x-o}$ and $0<t\le1$,
\begin{equation*}
 e^{-\psi(H_{o,t}(x))}
 \ge e^{-(1-t)r}e^{-\psi(x)}.
\end{equation*}
For $r>0$, convexity of $\sinh$ and its exponential formula give,
respectively,
\begin{equation*}
 \frac{\sinh(tr)}{\sinh r}\le t,
 \quad
 \frac{\sinh(tr)}{\sinh r}\le e^{-(1-t)r}.
\end{equation*}
Consequently,
\begin{equation*}
 \beta_{-n,n+1}^{(t)}(r)
 =t\left(\frac{\sinh(tr)}{\sinh r}\right)^n
 \le t^n e^{-(1-t)r}.
\end{equation*}
The same inequality holds at $r=0$ by continuity.
Multiplication by the Jacobian $t^n$ and integration now give
\eqref{eq:2.5} for $m$.

Apply Corollary~\ref{cor:5.7} to $m$ and write
$m=\int m_p\dd\vartheta(p)$. For $\vartheta$-almost every $p$,
with $k=\dim F_p$,
\begin{equation*}
 m_p\sim\Hh^k|_{F_p\cap X},\quad
 \supp m_p=F_p\cap X,\quad
 m_p(\ri F_p\cap\Omega)=1.
\end{equation*}
Extend $h$ by zero outside $\Omega$, and put
\begin{equation*}
 a=\mathbf{1}_\Omega h e^\psi,\qquad A(p)=\int a\dd m_p.
\end{equation*}
Since $d\mu=a\dd m$, Lemma~\ref{lem:2.2} gives
\begin{equation*}
 d\eta=A\dd\vartheta,\quad
 d\mu_p=\frac{a}{A(p)}\dd m_p
 \quad\text{for }\eta\text{-almost every }p.
\end{equation*}
Here $0<A(p)<\infty$ for $\eta$-almost every $p$.
The equivalence and concentration assertions in \eqref{eq:5.12}
follow immediately.
If $h>0$ almost everywhere in $\Omega$, then $a>0$ $m$-almost
everywhere, since $\lambda(\partial X)=0$. Disintegration gives
$a>0$ $m_p$-almost everywhere for $\vartheta$-almost every $p$.
Thus $\eta\sim\vartheta$ and $\mu_p\sim m_p$ at almost every
label, proving $\supp\mu_p=F_p\cap X$.
\end{proof}
The same measure-class statement holds on the corresponding graph faces.

\begin{corollary}[Graph-face measure class]\label{cor:5.9}
Let $f:\R^n\to\R$ be finite and convex and let $G$ be its graph. For $p\in \R^n$ put $F_p=(\partial f)^{-1}(p)$, and for $x\in\R^n$ set $\Gamma_f(x)=(x,f(x))$. Let  $G_p=\Gamma_f(F_p)$ for $p\in\R^n$. If $\omega$ is a strictly positive Borel function on the graph of $f$ with $0<\int_{G}\omega\dd\Hh^n<\infty$, then the normalised conditionals of $\omega\,\Hh^n|_G$
with respect to $Df\circ\Gamma_f^{-1}$ are equivalent to
$\Hh^{\dim G_p}|_{G_p}$ for
$(Df\circ\Gamma_f^{-1})_\#(\omega\,\Hh^n|_G)$-almost every $p$. In particular, their measure class is independent of $\omega$. 
\end{corollary}
\begin{proof}
The local area formula gives
\begin{equation*}
\Hh^n|_{G}=(\Gamma_f)_\#\bigl(\sqrt{1+\norm{Df}^2}\,\lambda\bigr).
\end{equation*}
Thus the pullback of the weighted graph measure has the strictly positive integrable density $\omega(\Gamma_f)\sqrt{1+\norm{Df}^2}$. Apply Corollary~\ref{cor:5.8} with $\Omega=\R^n$. On $F_p$, the map $\Gamma_f$ is affine and injective, with a constant positive tangential Jacobian. Its pushforward therefore takes the class of $\Hh^{\dim F_p}|_{F_p}$ to that of $\Hh^{\dim G_p}|_{G_p}$.
\end{proof}
Corollary~\ref{cor:5.9} recovers the measure-equivalence theorem of
Caravenna and Daneri \cite[Theorem~3.3 and Remark~3.5]{r11}:
the conditional measures obtained by disintegrating Hausdorff
measure on the graph of a finite convex function along its faces
are equivalent to the Hausdorff measures of the corresponding
dimensions on those faces. The same conclusion for a strictly
positive integrable weight follows by the change-of-density
formula in Lemma~\ref{lem:2.2}.

Their argument also provides quantitative information beyond
measure equivalence. In particular, \cite[Remark~4.16,
Proposition~4.17, and equations~(4.70)--(4.72)]{r11} gives
positivity, local Lipschitz regularity and two-sided
endpoint-distance estimates for the geometric density appearing
in their local disintegration formula
\cite[Theorem~4.18]{r11}. After localisation and passage to
continuous representatives, these estimates give, in our notation,
\begin{equation*}
 g_{F_p}(o+t(x-o))\ge t^{n-k}g_{F_p}(x)
 \text{ for }o\in F_p,\ x\in\ri F_p,\ 0<t\le1,
\end{equation*}
on almost every fibre $F_p=(\partial f)^{-1}(p)$ of dimension $k$.
Combining this inequality with the ambient density distortion
in Lemma~\ref{lem:2.1}, the tangential Jacobian $t^k$, and
normalisation gives the $\MCP(\kappa,N)$ conclusion of
Corollary~\ref{cor:5.7}, with unchanged parameters.
This MCP inheritance is a consequence of their quantitative
estimates, rather than a theorem formulated in those terms
in \cite{r11}.

Our argument treats convex-gradient fibres as a special case of
the general maximal monotone theorem, Theorem~\ref{thm:5.5}.
It identifies inverse fibres with translated resolvent leaves
by Lemma~\ref{lem:5.2}, applies the isometric-leaf inheritance
theorem, Theorem~\ref{thm:4.1}, and passes to the limit using
Proposition~\ref{prop:5.4} and Lemma~\ref{lem:4.2}.

\section{Curvature-dimension condition inheritance in codimension one}\label{sec:curvature}
The distortion estimate of Theorem~\ref{thm:3.1} gives measure contraction property on every dimensional stratum. Curvature-dimension condition inheritance requires a stronger control of the geometric density. In codimension one, the local ordering of disjoint affine hypersurface sections forces this density to be affine, without any transverse regularity assumption. This gives curvature-dimension condition inheritance in codimension one and on every leaf stratum in ambient dimensions one and two.

\subsection{Weighted curvature}

We recall the Euclidean characterisation of $\CD(\kappa,N)$ for
positive $\mathcal C^2$ densities. 

Let $X$ be a nonempty closed convex subset of a Euclidean space,
let $E=\Aff X$, and put $k=\dim X\ge1$ and $\Omega=\ri X$.
Consider a finite measure
\begin{equation}\label{eq:6.1}
d\mu=e^{-\rho}\dd\Hh^k|_X
\text{ with }\rho\in\mathcal C^2(\Omega).
\end{equation}
All derivatives below are intrinsic to $E$. No regularity of
$\rho$ at the relative boundary of $X$ is assumed.

For $N\in(k,\infty]$, the $N$-Bakry--\'Emery tensor is
\begin{equation}\label{eq:6.2}
\Ric_{\mu,N}
=
\begin{cases}
D^2\rho-\dfrac{D\rho D\rho^*}{N-k},
    & k<N<\infty,\\[6pt]
D^2\rho, & N=\infty.
\end{cases}
\end{equation}
In this setting, $(X,\norm{\cdot},\mu)$ satisfies
$\CD(\kappa,N)$ if and only if
\begin{equation*}
\Ric_{\mu,N}(x)(v,v)\ge\kappa\norm v^2
\text{ for all }x\in\Omega,\ v\in E-E.
\end{equation*}
For $N=k$, the condition requires $\rho$ to be constant and
$\kappa\le0$; in that case we set $\Ric_{\mu,k}=0$.
A measure supported on a singleton satisfies $\CD(\kappa,N)$
for every $\kappa\in\R$ and $N\in[1,\infty]$.

The equivalence with the synthetic curvature-dimension condition
is the weighted Riemannian characterisation; see
\cite[Theorem~4.22]{r27}. For the closed convex supports used here,
necessity follows from the local argument in the relative
interior. For sufficiency, the same optimal-transport proof
applies on compact convex subsets of $\Omega$, and exhaustion
and stability give the assertion on $X$. The differential
argument requires only $\mathcal C^2$ regularity of $\rho$;
no regularity at the relative boundary is needed.
In particular, $\CD(0,\infty)$ is equivalent to log-concavity
of the density. Multiplication of $\mu$ by a positive constant
does not affect the condition.

\subsection{Affine density in codimension one}
\begin{theorem}[Affine density in codimension one]\label{thm:6.2}
Let $n\ge2$, let $u:\R^n\to\R^m$ be $1$-Lipschitz, and let $d\mu=h\dd\lambda|_X$ be finite, where $X$ is closed and convex with nonempty interior $\Omega$ and $0<h<\infty$ almost everywhere on $X$. For almost every leaf $\S$ of dimension $n-1$, there is an affine function $p_{\S}$ on $\Aff\S$, positive on $\ri\S\cap\Omega$ and nonnegative on $\S\cap X$, such that
\begin{equation}\label{eq:6.5}
d\mu_{\S}=hp_{\S}\dd\Hh^{n-1}|_{\S\cap X}.
\end{equation}
\end{theorem}
\begin{proof}
We first show that, for $\nu_0$-almost every $(n-1)$-dimensional
leaf, the geometric density $g_{\S}$ of Theorem~\ref{thm:3.1}
extends to an affine function on $\Aff\S$.

Fix a chart of positive Lebesgue measure from
Proposition~\ref{prop:3.4}. 
Since $V$ has codimension one, choose orthonormal coordinates
identifying $V$ with $\R^{n-1}\times\{0\}$ and $V^\perp$ with
the last coordinate axis. Every anchor has the form $(q_*,a)$;
we identify the compact anchor set with its scalar coordinates
$A\subset\R$ and relabel the chart accordingly.
The identity $P_VL_a=\Id_V$ implies that there is a unique
$\ell_a\in\R^{n-1}$ such that
\begin{equation*}
 L_av=(v,\ip{\ell_a}{v})
 \text{ for }v\in\R^{n-1}.
\end{equation*}
Writing the original anchor as $b=(q_*,a)$, we have
\begin{equation*}
 \Phi(b,q)=b+L_b(q-q_*)=(q_*,a)+(q-q_*,\ip{\ell_a}{q-q_*})=\bigl(q,a+\ip{\ell_a}{q-q_*}\bigr).
\end{equation*}
We henceforth label the chart by the scalar anchor coordinate $a$
and retain the notation $\Phi(a,q)$.
Consequently, the chart takes the form
\begin{equation*}
 \Phi(a,q)=\bigl(q,a+\ip{\ell_a}{q-q_*}\bigr)
 \text{ for }a\in A,\ q\in D.
\end{equation*}
Thus $a$ is the height of the anchor at $q_*$, while $\ell_a$
is the slope of the affine leaf section.

For $a<b$ in $A$, disjointness of the sections implies that
\begin{equation*}
 q\longmapsto b-a+\ip{\ell_b-\ell_a}{q-q_*}
\end{equation*}
is positive throughout $D$: it never vanishes there and is positive
at $q_*$. If $B(q_*,r)\subset D$, taking the infimum over this
ball gives
\begin{equation}\label{eq:6.6}
 \norm{\ell_b-\ell_a}\le r^{-1}|b-a|.
\end{equation}
Extend $a\mapsto\ell_a$ to a Lipschitz function on $\R$.
At almost every Lebesgue density point $a$ of $A$, its derivative
$D\ell_a$ exists and is independent of the extension. Set for $q\in D$
\begin{equation*}
     p_a(q)=1+\ip{D\ell_a}{q-q_*}.
\end{equation*}
Difference quotients along $A$ of the increasing functions
$a\mapsto a+\ip{\ell_a}{q-q_*}$ show that $p_a(q)\ge0$
for every $q\in D$. Since $p_a$ is affine and $p_a(q_*)=1$,
it is strictly positive on the open set $D$.

Put $Z=\Phi(A\times C)$. The map $\Phi$ is Lipschitz on
$A\times C$ and injective. The area formula \cite{r21} gives
\begin{equation}\label{eq:6.7}
 d(\Phi^{-1})_\#(\lambda|_Z)
 =p_a(q)\dd\bigl((\lambda^1|_A)\otimes(\lambda^{n-1}|_C)\bigr).
\end{equation}
Comparing with \eqref{eq:3.9}, where $\zeta=\lambda^1|_A$,
shows that $f(a,q)=p_a(q)$ almost everywhere. Both functions are
continuous in $q$ for almost every $a$, so the equality holds for
every $q\in C$ at those anchors. Equation~\eqref{eq:3.14} therefore
gives
\begin{equation}\label{eq:6.8}
 g_{\S_a}(\Phi(a,q))
 =\frac{\sigma_{\S_a}(Z)}{W(a)J_a}\,p_a(q),
 \quad J_a=\sqrt{1+\norm{\ell_a}^2}.
\end{equation}
Every nonempty chart section of a leaf for which
Theorem~\ref{thm:3.1} holds has positive $\sigma_{\S}$-mass.
Thus \eqref{eq:3.12} transfers the exceptional anchors to a
$\nu_0$-negligible family of leaves. Negligible chart images can also be
discarded. Since the charts are countable and cover the relative
interior of almost every leaf, $g_{\S}$ is locally affine on
$\ri\S$ for $\nu_0$-almost every leaf of dimension $n-1$.
Connectedness of $\ri\S$ shows that these local affine expressions
are restrictions of a single affine function on $\Aff\S$.
Its extension is nonnegative on $\S$ by continuity.

Finally, $\nu=\S_\#\mu\ll\nu_0$. For $\nu$-almost every leaf
of dimension $n-1$, Theorem~\ref{thm:3.1} gives
\begin{equation*}
 0<m_{\S}=\int_{\S\cap X}h g_{\S}\dd\Hh^{n-1}<\infty.
\end{equation*}
Define $p_{\S}=g_{\S}/m_{\S}$ on $\Aff\S$.
It is affine, positive on $\ri\S\cap \Omega$ and nonnegative on
$\S\cap X$, and \eqref{eq:3.3} gives \eqref{eq:6.5}.
\end{proof}

\begin{corollary}[Curvature-dimension condition inheritance in codimension one]\label{cor:6.3}
In Theorem~\ref{thm:6.2}, suppose that $h=e^{-\rho}$ on $\Omega$,
where $\rho\in\mathcal C^2(\Omega)$, and that
$(X,\norm{\cdot},\mu)$ satisfies $\CD(\kappa,N)$ for
$\kappa\in\R$ and $N\in[n,\infty]$. Then almost every
$(n-1)$-dimensional leaf, equipped with its conditional measure
and the Euclidean metric, satisfies $\CD(\kappa,N)$ on $\S\cap X$.
\end{corollary}

\begin{proof}
Since $\mu(\Omega)=\mu(X)$ and the conditional measures give
full mass to the relative interiors of their leaves,
$\mu_{\S}(\ri\S\cap\Omega)=1$ for almost every leaf. Fix  an $(n-1)$-dimensional leaf $\mathcal{S}$.
 Let $V_{\S}=\Aff\S-\Aff\S$ denote the tangent space of $\S$, and write $W=\rho-\log p_{\S}$ on $\ri\S\cap\Omega$. Suppose first that $n<N<\infty$, and put $s=N-n$. For $v\in V_{\S}$, let $\alpha=D\rho(v)$ and $\beta=D\log p_{\S}(v)$. Since $p_{\S}$ is affine, we have
\begin{equation*}
\Ric_{\mu_{\S},N}(v,v)=D^2\rho(v,v)+\beta^2-\frac{(\alpha-\beta)^2}{s+1}=\Ric_{\mu,N}(v,v)+\frac{(\alpha+s\beta)^2}{s(s+1)}\ge\kappa\norm v^2.
\end{equation*}
For $N=n$, the ambient condition forces $\rho$ to be constant and $\kappa\le0$; then $\Ric_{\mu_{\S},n}=0$. For $N=\infty$,
\begin{equation*}
D^2_{\S}W=D^2_{\S}\rho+D_{\S}\log p_{\S} (D_{\S}\log p_{\S})^*\ge\kappa\Id.
\end{equation*}
The preceding inequalities establish the required
Bakry--\'Emery bound on $\ri S\cap \Omega$.
Since $\S\cap X$ is closed and convex, $W\in\mathcal C^2(\ri \S\cap X)$,
and $\mu_{\S}=e^{-W}\Hh^{n-1}|_Y$, the weighted Euclidean
characterisation above gives $\CD(\kappa,N)$ on $\S\cap X$.
\end{proof}

\begin{corollary}[Ambient dimensions one and two]\label{cor:6.4}
Let $n\le2$, let $X\subset\R^n$ be closed and convex with
nonempty interior $\Omega$, and let
$d\mu=e^{-\rho}\dd\lambda|_X$ be finite, with
$\rho\in\mathcal C^2(\Omega)$. Suppose that
$(X,\norm{\cdot},\mu)$ satisfies $\CD(\kappa,N)$ for
$\kappa\in\R$ and $N\in[n,\infty]$.
For every $1$-Lipschitz map $u:\R^n\to\R^m$,
almost every leaf, equipped with its conditional measure
and the Euclidean metric, satisfies $\CD(\kappa,N)$ on $\S\cap X$.
\end{corollary}
\begin{proof}
For $n=2$, one-dimensional leaves are covered by
Corollary~\ref{cor:6.3}. Full-dimensional leaves carry normalised
restrictions of the ambient measure to convex sets, while
zero-dimensional leaves carry Dirac masses. These observations
also prove the assertion for $n=1$.
\end{proof}

Thus failure of curvature inheritance requires
positive-dimensional leaves of codimension at least two.
The same conclusion holds without the regularity assumption
on the ambient density, by the extension discussed in the
introduction. The three-dimensional example in
Section~\ref{sec:failure} shows that the ambient dimension
threshold is sharp.

\section{Failure of curvature inheritance}\label{sec:failure}
We shall show that the curvature-dimension condition need not pass to the conditional measures on leaves, even when the map is firmly nonexpansive or is a gradient. In both examples the leaves have codimension two. In the first, the same decomposition is also given by the fibres of a Lipschitz monotone map. In the second, the potential is nonconvex, and its Hessian has both positive and negative eigenvalues along each leaf.

We use the local necessity of the Bakry--\'Emery bound recalled
in Section~\ref{sec:curvature}.
Let $\eta$ have closed convex support $Y$ of dimension $k$,
let $U\subset\ri Y$ be relatively open, and suppose that
$\eta|_U=h\Hh^k|_U$ with $h\in\mathcal C^2(U)$ and $h>0$.

If $(Y,\norm{\cdot},\eta)$ satisfies $\CD(\kappa,N)$ with
$k<N<\infty$, then
\begin{equation}\label{eq:7.1}
-D^2\log h-\frac{D\log h (D\log h)^*}{N-k}
\ge\kappa\Id
\text{ on }U.
\end{equation}
The derivatives are intrinsic to $\Aff Y$.
For $N=\infty$, the second term is omitted.
In particular, a positive directional second derivative of
$\log h$ at a point of $U$ excludes $\CD(0,N)$ for every
$N\in(k,\infty]$.

\subsection{A firmly nonexpansive map in dimension three}
\begin{theorem}\label{thm:7.1}
There exist a firmly nonexpansive map $u:\R^3\to\R^3$ and a closed Euclidean ball $X$ with nonempty interior such that normalised Lebesgue measure on $X$ satisfies $\CD(0,3)$, whereas almost every conditional measure in its leaf disintegration is supported on a one-dimensional leaf and fails $\CD(0,N)$ for every $N\in(1,\infty]$.

The map $u$ is the resolvent of an everywhere defined Lipschitz monotone map. Moreover, its leaf decomposition is also the fibre decomposition of an everywhere defined Lipschitz monotone map.
\end{theorem}
\begin{proof}
Let $C=\cl B(0,1)\subset\R^3$, let $b=e_3$, and define the skew-symmetric linear map $B$ by $Bp=b\times p=(-p_2,p_1,0)$ for $p\in C$. For $p\in C$, let
\begin{equation*}
   N_C(p)=\{v\in \R^3\mid \ip{v}{q-p}\le0\text{ for all }q\in C\}
\end{equation*}
be the outward normal cone to $C$. 
The relation $B+N_C$ is monotone because $N_C$ is monotone and $\langle Bh,h\rangle=0$ for every $h\in\mathbb{R}^3$. It is also maximally monotone. Indeed, suppose that $(p,z)\in \R^3\times \R^3$ is monotonically related to its graph, so that $\langle z-Bq-v,p-q\rangle\geq0$ for every $q\in C$ and $v\in N_C(q)$. Setting $w=z-Bp$, skew-symmetry gives $\langle w-v,p-q\rangle\geq0$. Let 
\begin{equation*}
    q=P_C(p+w)\text{ and }v=p+w-q,
\end{equation*}
where $P_C$ denotes the metric projection onto $C$. The definition of the metric projection ensures that $v\in N_C(q)$, and hence 
\begin{equation*}
    0\leq\langle w-v,p-q\rangle=-\norm{p-q}^2.
\end{equation*}
Thus $p=q$ and $w=v\in N_C(p)$, which yields $z\in Bp+N_C(p)$. Therefore $B+N_C$ admits no proper monotone extension and is maximally monotone.

For $a>0$, put
\begin{equation*}
M_a=a\Id+B+N_C,\quad T_a=M_a^{-1}.
\end{equation*}
Minty's theorem shows that $M_a$ is onto; see \cite{r4}. If $x\in M_a(p)$ and $y\in M_a(q)$, then
\begin{equation}\label{eq:7.2}
\ip{x-y}{p-q}\ge a\norm{p-q}^2.
\end{equation}
Thus $T_a$ is everywhere defined, single-valued, monotone and $1/a$-Lipschitz. Its nonempty fibres are for $p\in C$
\begin{equation}\label{eq:7.3}
T_a^{-1}(p)=\begin{cases}
\{Bp+ap\},&\norm p<1,\\
\{Bp+sp\mid s\ge a\},&\norm p=1.
\end{cases}
\end{equation}
Fix $\tau>0$, put $c=a+\tau$, and let
\begin{equation*}
u=(\Id+\tau T_a)^{-1}.
\end{equation*}
Since $T_a$ is continuous and monotone on all of $\R^3$, it is
maximal monotone \cite{r4}. Thus $u$ is firmly nonexpansive, and
Lemma~\ref{lem:5.2} identifies its leaves as
$\tau p+T_a^{-1}(p)$. By \eqref{eq:7.3}, the positive-dimensional
leaves are precisely
\begin{equation}\label{eq:7.5}
 \S_p=\{Bp+sp\mid s\ge c\},\text{ when }\norm{p}=1.
\end{equation}
On $\S_p$ one has $u(z)=z-\tau p$; the other leaves are singletons.
Equation~\eqref{eq:7.3}, shows that
these leaves are precisely the fibres of the everywhere
defined Lipschitz monotone map $T_a$.

Let $S^2\subset\R^3$ denote the two-dimensional unit sphere. To compute the conditional measures, consider the parametrisation
\begin{equation*}
F: S^2\times(c,\infty)\to\R^3,\quad F(p,s)=Bp+sp.
\end{equation*}
It is injective, since $T_c(F(p,s))=p$. Choose an oriented orthonormal basis $(v_1,v_2)$ of $T_p S^2$. Since the derivative of $F$ with respect to the second variable is $p$, its Jacobian relative to $\Hh^2|_{S^2}\otimes \lambda^1$ is
\begin{equation}\label{eq:7.6}
J_F(p,s)=\det(Bv_1+sv_1,Bv_2+sv_2,p)=\det\begin{pmatrix}s&-\ip bp\\\ip bp&s\end{pmatrix}
=s^2+\ip bp^2.
\end{equation}
Thus $F$ is a smooth diffeomorphism onto an open subset of $\R^3$.

Take $a=\tau=1/8$, so $c=1/4$, and put $x_*=F(b,1/2)=b/2$. The open set
\begin{equation*}
\mathcal V=\{(p,s)\in S^2\times (c,\infty)\mid \ s^2<\ip bp^2\}
\end{equation*}
contains $(b,1/2)$. Choose a closed ball $X$ of positive radius centred at $x_*$ and contained in $F(\mathcal V)$, and put $\mu=\lambda(X)^{-1}\lambda|_X$. This measure satisfies $\CD(0,3)$.

For $p\in S^2$, let $I_p=\{s\in (c,\infty)\mid F(p,s)\in X\}$. This is a compact interval or the empty set. Change of variables gives the quotient density
\begin{equation*}
\frac{1}{\lambda(X)}\int_{I_p}\bigl(s^2+\ip bp^2\bigr)\dd \lambda^1(s)
\end{equation*}
with respect to $\Hh^2|_{ S^2}$. Thus the degenerate intervals have zero quotient measure. Since $s\mapsto F(p,s)$ has unit speed, the normalised conditional on every nondegenerate interval has density
\begin{equation}\label{eq:7.7}
h_p(s)=\frac{s^2+\ip bp^2}{\displaystyle\int_{I_p}(r^2+\ip bp^2)\dd \lambda^1(r)}.
\end{equation}
The choice of $X\subset F(\mathcal V)$ ensures that, at every interior point,
\begin{equation}\label{eq:7.8}
(\log h_p)''(s)=\frac{2(\ip bp^2-s^2)}{(s^2+\ip bp^2)^2}>0.
\end{equation}
It follows from \eqref{eq:7.1} that each of these conditionals fails $\CD(0,N)$ for every $N\in(1,\infty]$.
\end{proof}

\begin{corollary}[Failure of curvature-dimension condition inheritance on monotone fibres]
\label{cor:monotone-cd-failure}
There exist an everywhere-defined firmly nonexpansive map
$T:\R^3\to\R^3$ and a closed Euclidean ball $X$ with nonempty
interior such that $\mu=\lambda(X)^{-1}\lambda|_X$ satisfies
$\CD(0,3)$, whereas, in its disintegration
$\mu=\int\mu_p\dd\eta(p)$ with respect to $T$, almost every
conditional has support $T^{-1}(p)\cap X$ of dimension one and
fails $\CD(0,N)$ for every $N\in(1,\infty]$.
In particular, curvature-dimension bounds are not inherited by
fibres of general maximal monotone operators.
\end{corollary}

\begin{proof}
With the notation of the preceding proof, set $T=aT_a$.
Inequality~\eqref{eq:7.2} shows that $T$ is firmly nonexpansive;
being everywhere defined and continuous, it is also maximally
monotone. Since $X\subset F(\mathcal V)$ and $s>c>a$ on
$\mathcal V$, equation~\eqref{eq:7.3} gives
\begin{equation*}
T^{-1}(ap)\cap X=\S_p\cap X
\text{ for every }p\in S^2.
\end{equation*}
The conditional measures therefore coincide, up to relabelling,
with those computed in the preceding proof, which establishes
the conclusion.
\end{proof}

\begin{remark}\label{rem:7.2}
The map $u$ constructed above is not a gradient: on the open
ellipsoid $(B+c\Id)(\ri C)$, its derivative is the nonsymmetric
matrix $\Id-\tau(B+c\Id)^{-1}$. The skew term also produces the
curvature failure in this example. Indeed, if $B=0$, the
conditional density in \eqref{eq:7.7} is proportional to $s^2$,
so its square root is affine and the conditionals satisfy
$\CD(0,3)$.
\end{remark}

\subsection{A gradient contraction in dimension four}
\begin{theorem}\label{thm:7.5}
There exist a function $\varphi\in\mathcal C^{1,1}(\R^4)$ with
$1$-Lipschitz gradient $u=D\varphi$ and a closed Euclidean ball $X$
with nonempty interior such that normalised Lebesgue measure on $X$
satisfies $\CD(0,4)$, whereas almost every conditional measure in its
leaf disintegration is supported on a two-dimensional leaf and fails
$\CD(0,N)$ for every $N\in(2,\infty]$.

The map $u$ is smooth near $X$. At every point of each leaf section
in $X$, the restriction of $D^2\varphi$ to the tangent plane has
eigenvalues $+1$ and $-1$.
\end{theorem}

\begin{proof}
Put for $(r,z,s,w)\in\R^4$
\begin{equation*}
 Q(r,z,s,w)=\frac12(r^2+z^2-s^2-w^2).
\end{equation*}
Its gradient is the reflection $(r,z,s,w)\mapsto(r,z,-s,-w)$.
For $\xi=(r,z,s,w)\in\R^4$ with $r,s>0$, introduce the quadratic saddle function 
\begin{equation*}
 \Psi_\xi(p,q)=\frac r2p^2-\frac s2q^2+2pq-zp+wq,
 \text{ and }
 \varphi_0(\xi)=Q(\xi)+\min\{\max\{\Psi_\xi(p,q)\mid q\in\R\}\mid p\in\R\}
\end{equation*}
The saddle point is unique and it satisfies
\begin{equation*}
 z=rp+2q,\qquad w=sq-2p.
\end{equation*}
Let $\Omega=\{(r,z,s,w)\in\R^4\mid r,s>0\}$. The map $F\colon \R^2\times (0,\infty)^2\to \Omega$ defined by
\begin{equation*}
 F(p,q,r,s)=(r,rp+2q,s,sq-2p)\text{ for }(p,q,r,s)\in \R^2\times (0,\infty)^2
\end{equation*}
gives the coordinates of these stationary points and is a diffeomorphism,
since $rs+4>0$, and thus these equations determine $(p,q)$ smoothly and
uniquely from $\xi$. The parameters $(p,q)$ label affine plane sections, while
$(r,s)$ are coordinates within each section. Below we identify
these sections locally with the leaves of the global extension.

For $(p(\xi),q(\xi))$ denoting the saddle point, 
\begin{equation*}
 u_0(\xi)=D\varphi_0(\xi)
 =(r,z,-s,-w)+\left(\frac{p(\xi)^2}{2},-p(\xi),-\frac{q(\xi)^2}{2},q(\xi)\right).
\end{equation*}
For fixed $(p,q)$, the correction to the reflection $DQ$ is constant.
Thus, on the affine plane section parametrised by
\begin{equation*}
 (0,\infty)^2\ni(r,s)\longmapsto F(p,q,r,s)\in\Omega,
\end{equation*}
the map $u_0$ is the restriction of an affine isometry of $\R^4$.
Explicitly,
\begin{equation*}
 u_0(F(p,q,r,s))
 =r(1,p,0,0)-s(0,0,1,q)
 +\left(\frac{p^2}{2},2q-p,-\frac{q^2}{2},2p+q\right).
\end{equation*}
The orthogonal tangent vectors of that plane section are
\begin{equation*}
 v_+=(1,p,0,0),\qquad v_-=(0,0,1,q),
\end{equation*}
and the above gives
\begin{equation*}
 D^2\varphi_0(F)v_+=v_+,
 \qquad D^2\varphi_0(F)v_-=-v_-.
\end{equation*}

Put $\xi_*=(1,0,1,0)=F(0,0,1,1)$ and
$H(\xi)=D^2\varphi_0(\xi)$. At $\xi_*$, the two eigenvalues corresponding to eigenvectors of $H$
orthogonal to $v_+,v_-$ are $\pm2/\sqrt5$, strictly between
$-1$ and $1$. The matrix $H$ is symmetric, and its tangent
eigenvalues remain exactly $+1,-1$ by the preceding identities.
Continuity therefore gives $\rho>0$ such that
$B(\xi_*,3\rho)\subset\Omega$ and, throughout this ball,
\begin{equation*}
 \norm{H}\le1,\qquad
 \ker(\Id-H^2)=\operatorname{span}\{v_+,v_-\},
 \qquad r^2+s^2<8.
\end{equation*}
In particular, $u_0$ is $1$-Lipschitz on this ball.

We extend $\varphi_0$ from $E=\cl B(\xi_*,\rho)$.
For $x,y\in E$, the points
\begin{equation*}
 z_\pm=\frac{x+y}{2}\pm\frac{u_0(x)-u_0(y)}2
\end{equation*}
lie in $B(\xi_*,3\rho)$. 
Set
\begin{equation*}
 R=\varphi_0(x)-\varphi_0(y)
   -\frac12\ip{u_0(x)+u_0(y)}{x-y}.
\end{equation*}
The bound $\norm{D^2\varphi_0}\le1$ on the larger
ball gives the quadratic Taylor estimates
\begin{equation*}
 \varphi_0(x)+\ip{u_0(x)}{z_+-x}
       -\frac12\norm{z_+-x}^2
 \le \varphi_0(z_+)
 \le \varphi_0(y)+\ip{u_0(y)}{z_+-y}
       +\frac12\norm{z_+-y}^2.
\end{equation*}
Comparison of the two outer expressions yields
\begin{equation*}
 R\le\frac14\bigl(\norm{x-y}^2-\norm{u_0(x)-u_0(y)}^2\bigr).
\end{equation*}
Interchanging $x$ and $y$ replaces $R$ by $-R$ and $z_+$ by
$z_-$, so the same argument gives the corresponding lower bound.
Consequently, for all $x,y\in E$
\begin{equation*}
 \left|\varphi_0(x)-\varphi_0(y)
   -\frac12\ip{u_0(x)+u_0(y)}{x-y}\right|
 \le\frac14\left(
   \norm{x-y}^2-\norm{u_0(x)-u_0(y)}^2\right).
\end{equation*}
Le Gruyer's  extension theorem \cite{legruyer2009} thus
gives $\varphi\in\mathcal C^{1,1}(\R^4)$ agreeing with
$\varphi_0$ and its gradient on $E$, with
$\operatorname{Lip}(D\varphi)\le1$.
Set $u=D\varphi$ and $X=\cl B(\xi_*,\rho/2)$.
Then $u$ is smooth near $X$, and normalised Lebesgue measure
$\mu=\lambda(X)^{-1}\lambda|_X$ satisfies $\CD(0,4)$.

For each label $(p,q)$ whose plane meets $X$, put
\begin{equation*}
 P_{p,q}=(0,2q,0,-2p)+\operatorname{span}\{v_+,v_-\}.
\end{equation*}
The map $u$ is an affine isometry on $P_{p,q}\cap E$.
By the leaf geometry recalled in Section~\ref{sec:geometry},
this set is contained in a unique leaf $\S_{p,q}$.
If $\xi\in P_{p,q}\cap X$ and $\eta\in\S_{p,q}$,
differentiating the affine isometry along $[\xi,\eta]$ gives
\begin{equation*}
 \norm{H(\xi)(\eta-\xi)}=\norm{\eta-\xi}.
\end{equation*}
The kernel identity above forces $\eta-\xi$ to lie in
$\operatorname{span}\{v_+,v_-\}$. Thus $\S_{p,q}\subset P_{p,q}$,
and
\begin{equation*}
 \dim\S_{p,q}=2,\qquad
 \S_{p,q}\cap X=P_{p,q}\cap X.
\end{equation*}
Distinct labels give distinct leaves, so these are precisely
the leaf sections in $X$.

The tangential Jacobian 
\begin{equation*}
    \big(|\det\big((v_+,v_-)^*(v_+,v_-)\big)|\big)^\frac12=\big((1+p^2)(1+q^2)\big)^\frac12
\end{equation*}
is constant on each plane. The ambient Jacobian $|\det DF(p,q,r,s)|=rs+4$,
change of variables, and Fubini's theorem therefore give
the conditional densities
\begin{equation*}
 h_{p,q}(F(p,q,r,s))=c_{p,q}(rs+4),
 \quad c_{p,q}>0,
\end{equation*}
with respect to $\Hh^2|_{\S_{p,q}\cap X}$, for almost every
leaf. Here $c_{p,q}$ is a constant that normalises the mass. Sections of zero
area have zero quotient measure, and every remaining section
is a closed planar disk with full conditional support.

In the unit tangent direction
\begin{equation*}
 e=\frac{v_++v_-}{\sqrt{2+p^2+q^2}},
\end{equation*}
we obtain
\begin{equation}\label{eq:rational-log-curvature}
 D^2\log h_{p,q}(e,e)
 =\frac{8-r^2-s^2}{(2+p^2+q^2)(rs+4)^2}>0.
\end{equation}
The curvature criterion \eqref{eq:7.1} excludes
$\CD(0,N)$ for every $N\in(2,\infty]$.
The exact tangent eigenvalue identities prove the remaining
assertion.
\end{proof}

\subsection{Arbitrarily large curvature loss}
Rescaling the preceding examples shows that the failure of curvature-dimension condition inheritance cannot be repaired by allowing a fixed loss in the lower curvature bound.

\begin{corollary}\label{cor:7.7}
Let $A>0$ and $\epsilon>0$. There exist:
\begin{enumerate}
\item a firmly nonexpansive map on $\R^3$ and a closed ball of diameter less than $\epsilon$ such that almost every conditional measure in the leaf disintegration of normalised Lebesgue measure on the ball is supported on a one-dimensional leaf and fails $\CD(-A,N)$ for every $N\in(1,\infty]$;
\item a gradient contraction $D\varphi$ on $\R^4$, with $\varphi\in \mathcal C^{1,1}(\R^4)$, and a closed ball of diameter less than $\epsilon$ such that almost every conditional measure in the leaf disintegration of normalised Lebesgue measure on the ball is supported on a two-dimensional leaf and fails $\CD(-A,N)$ for every $N\in(2,\infty]$.
\end{enumerate}
Alternatively, the ambient measure may be fixed in advance to be either normalised Lebesgue measure on the unit ball or the standard Gaussian in the corresponding dimension. In that case the asserted failure holds on a family of leaves of positive quotient measure.
\end{corollary}
\begin{proof}
In either construction, let $F(a,z)$ denote the local chart, with $a$ labelling the leaves and $z$ the affine coordinates along each leaf. Its geometric density is
\begin{equation*}
    j(a,z)=\frac{|\det DF(a,z)|}{J_k(D_zF(a,z))}.
\end{equation*}
Here $(n,k)=(3,1)$ or $(4,2)$, and the tangential Jacobian is constant along each leaf. By \eqref{eq:7.8} and \eqref{eq:rational-log-curvature}, there are an interior point $x_*$, an open subchart, a continuous unit tangent field $v$, and $c_*>0$ such that
\begin{equation}\label{eq:7.24}
D^2_{\S}\log j(v,v)\ge c_*
\end{equation}
throughout that subchart. In the first example, the value at $(p,s)=(e_3,1/2)$
is $24/25$; for the second, the value at $(r,z,s,w)=(1,0,1,0)$ is $3/25$.

For $\delta>0$, define
\begin{equation*}
    u_\delta(x)=\delta\bigl(u(x_*+x/\delta)-u(x_*)\bigr).
\end{equation*}
Its leaves are exactly the sets $\delta(\S-x_*)$, where $\S$ runs through the leaves of $u$. Both the $1$-Lipschitz property and firm nonexpansiveness are preserved by this rescaling. In the gradient case,
\begin{equation*}
  u_\delta=D\varphi_\delta,\qquad
\varphi_\delta(x)=\delta^2\varphi(x_*+x/\delta)-\delta\ip{D\varphi(x_*)}{x}.  
\end{equation*}
The scaled chart is $F_\delta=\delta(F-x_*)$, with geometric density $\delta^{n-k}j$. Hence the corresponding unit-direction second derivative of its logarithm is at least $c_*\delta^{-2}$.

Choose a closed ball $X_0$ centred at $x_*$ and contained in the image of the subchart. For normalised Lebesgue measure on $X_\delta=\delta(X_0-x_*)$, almost every conditional measure has density proportional to $\delta^{n-k}j$. Take $\delta$ so small that $\operatorname{diam}X_\delta<\epsilon$ and $c_*\delta^{-2}>A$. Equation~\eqref{eq:7.1} then excludes every curvature-dimension bound in the statement.

For the last assertion, let the fixed ambient density be proportional to $e^{-\rho}$, with $\rho=0$ on the unit ball or $\rho(x)=\norm x^2/2$ for $x\in\R^n$. On the image of the scaled subchart, the conditional measure on a leaf of $u_\delta$ has density
\begin{equation}\label{eq:7.25}
h_a(F_\delta(a,z))=C(a)e^{-\rho(F_\delta(a,z))}\delta^{n-k}j(a,z),\quad C(a)>0.
\end{equation}
To verify this formula for the disintegration of the full ambient measure, restrict that measure to the image of the subchart and apply change of variables. Uniqueness of disintegration identifies each restricted conditional with the restriction of its global conditional, up to a positive constant depending only on its leaf.

Since $D^2\rho\le\Id$, equations~\eqref{eq:7.24} and~\eqref{eq:7.25} give, in the indicated unit direction,
\begin{equation*}
    D^2_{\S}\log h_a(v,v)\ge c_*\delta^{-2}-1.
\end{equation*}
Choose $\delta$ so that this exceeds $A$ and so that the image of the subchart lies inside the unit ball when required. This image has positive ambient measure. Therefore the leaves on which the displayed local density formula and curvature-dimension condition obstruction hold have positive quotient measure. 
\end{proof}

\begin{remark}\label{rem:7.8}
The ambient unit ball satisfies $\CD(0,n)$, and the standard Gaussian satisfies $\CD(1,\infty)$. Thus no curvature lower bound depending only on the ambient curvature and dimension can hold for the indicated leaf conditionals. The map in Corollary~\ref{cor:7.7} may depend on $A$; the statement does not assert unbounded curvature loss for a single map.
\end{remark}

\section*{Table of notation}

Conditional measures have mass one; references point to the detailed definitions.

\begingroup
\footnotesize
\renewcommand{\arraystretch}{0.96}
\setlength{\LTleft}{0pt}
\setlength{\LTright}{0pt}
\begin{longtable}{@{}p{0.13\textwidth}@{\hspace{0.01\textwidth}}p{0.345\textwidth}@{\hspace{0.03\textwidth}}p{0.13\textwidth}@{\hspace{0.01\textwidth}}p{0.345\textwidth}@{}}
\textit{notation} & \textit{meaning} & \textit{notation} & \textit{meaning} \\[3pt]
\endhead
$\norm{\cdot}$ & \mbox{the Euclidean norm} & $\nu_0$ & \mbox{reference quotient $\S_\#\sigma$} \tabularnewline
$\ip{\cdot}{\cdot}$ & \mbox{the Euclidean scalar product} & $\sigma_{\S}$ & \mbox{reference leaf conditional} \tabularnewline
$\Id$ & \mbox{identity map} & $g_{\S}$ & \mbox{geometric density (Thm.~\ref{thm:3.1})} \tabularnewline
$L^*$ & \mbox{adjoint of $L$} & $m_{\S}$ & \mbox{leaf mass $\int_{\S}hg_{\S}\dd\Hh^k$} \tabularnewline
$P_V$ & \mbox{orthogonal projection onto $V$} & $h_{\S}$ & \mbox{leaf density $hg_{\S}/m_{\S}$} \tabularnewline
$B(x,r)$ & \mbox{open ball with centre $x$, radius $r$} & $\lambda_{\S}$ & \mbox{geometric measure $g_{\S}\Hh^k|_{\S}$} \tabularnewline
$\omega_n$ & \mbox{unit-ball volume in $\R^n$} & $\T^{k,0}$ & \mbox{conull subfamily of $\T^k$ (Prop.~\ref{prop:3.4})} \tabularnewline
$\Aff E$ & \mbox{affine hull of $E$} & $\Phi$ & \mbox{affine leaf chart \eqref{eq:3.7}} \tabularnewline
$\Conv E$ & \mbox{convex hull of $E$} & $L_a$ & \mbox{linear part of the chart at $a$} \tabularnewline
$\ri C$ & \mbox{relative interior of $C$} & $\Sigma_q$ & \mbox{transverse slice $\Phi(A\times\{q\})$} \tabularnewline
$\cl C$ & \mbox{relative closure of $C$} & $\zeta$ & \mbox{anchor measure $\Hh^{n-k}|_A$} \tabularnewline
$\partial C$ & \mbox{relative boundary of $C$} & $f(a,q)$ & \mbox{chart density \eqref{eq:3.9}} \tabularnewline
$DF$ & \mbox{differential of $F$} & $W(a)$ & \mbox{chart normalising mass (Lem.~\ref{lem:3.5})} \tabularnewline
$D^2F$ & \mbox{second derivative of $F$} & $J_a$ & \mbox{leaf Jacobian $\sqrt{\det(L_a^*L_a)}$} \tabularnewline
$D_{\S}$ & \mbox{intrinsic derivative on $\Aff\S$} & $J_k(L)$ & \mbox{$k$-dimensional Jacobian of $L$} \tabularnewline
$D^2_{\S}$ & \mbox{intrinsic Hessian on $\Aff\S$} & $\ell_\pm(x,v)$ & \mbox{chord endpoint distances (Cor.~\ref{cor:3.8})} \tabularnewline
$\mathcal C^{1,1}$ & \mbox{functions with Lipschitz derivative} & $H_{o,t}$ & \mbox{homothety $x\mapsto o+t(x-o)$} \tabularnewline
$\Lip_{\mathrm{loc}}$ & \mbox{locally Lipschitz functions} & $s_a$ & \mbox{model sine (\textsection\,\ref{sec:geometry})} \tabularnewline
$\lambda$ & \mbox{ambient Lebesgue measure} & $R_{\kappa,N}$ & \mbox{admissible contraction radius} \tabularnewline
$\lambda^d$ & \mbox{$d$-dimensional Lebesgue measure} & $\beta_{\kappa,N}^{(t)}$ & \mbox{distortion coefficient \eqref{eq:2.4}} \tabularnewline
$\Hh^k$ & \mbox{$k$-dimensional Hausdorff measure} & $\MCP(\kappa,N)$ & \mbox{measure contraction property} \tabularnewline
$\mu|_E$ & \mbox{restriction of $\mu$ to $E$} & $A$ & \mbox{maximal monotone relation} \tabularnewline
$F_\#\mu$ & \mbox{pushforward of $\mu$ by $F$} & $T$ & \mbox{a Borel selection of $A$} \tabularnewline
$\mu\ll\eta$ & \mbox{absolute continuity} & $\dom A$ & \mbox{domain of $A$} \tabularnewline
$\mu\sim\eta$ & \mbox{mutual absolute continuity} & $\ran A$ & \mbox{range of $A$} \tabularnewline
$\supp\mu$ & \mbox{support of $\mu$} & $M_A$ & \mbox{multivalued locus $\{x:\#A(x)>1\}$} \tabularnewline
$\norm{\gamma}_{\TV}$ & \mbox{total variation $|\gamma|(Z)$} & $J_\epsilon$ & \mbox{resolvent $(\Id+\epsilon A)^{-1}$} \tabularnewline
$\mathbf1_E$ & \mbox{indicator of $E$} & $A_\epsilon$ & \mbox{the Yosida approximation of $A$} \tabularnewline
$X$ & \mbox{ambient closed convex set} & $F_p$ & \mbox{inverse fibre $A^{-1}(p)$} \tabularnewline
$\Omega$ & \mbox{interior of $X$} & $Y_p$ & \mbox{fibre section $F_p\cap X$} \tabularnewline
$\mu$ & \mbox{ambient finite measure} & $\S_{\epsilon,p}$ & \mbox{resolvent leaf $\epsilon p+F_p$} \tabularnewline
$h$ & \mbox{ambient density with respect to $\lambda$} & $\eta$ & \mbox{fibre quotient $T_\#\mu$} \tabularnewline
$L_h$ & \mbox{the Lebesgue points of $h$} & $\mu_p$ & \mbox{fibre conditional probability} \tabularnewline
$n$ & \mbox{ambient dimension} & $\eta_\epsilon$ & \mbox{quotient $(A_\epsilon)_\#\mu$} \tabularnewline
$m$ & \mbox{target dimension} & $\mu_p^\epsilon$ & \mbox{conditional for $A_\epsilon$} \tabularnewline
$k$ & \mbox{leaf or fibre dimension} & $t_p^\epsilon$ & \mbox{translation $x\mapsto x-\epsilon p$} \tabularnewline
$d$ & \mbox{leaf codimension $n-k$} & $\widehat\mu_p^\epsilon$ & \mbox{translated conditional $(t_p^\epsilon)_\#\mu_p^\epsilon$} \tabularnewline
$\kappa$ & \mbox{curvature parameter} & $G$ & \mbox{monotone graph $\graph A$ (Prop.~\ref{prop:5.4})} \tabularnewline
$N$ & \mbox{synthetic dimension parameter} & $D_GF$ & \mbox{tangential differential on $G$} \tabularnewline
$u$ & \mbox{nonexpansive map $\R^n\to\R^m$} & $J_GF$ & \mbox{tangential Jacobian on $G$} \tabularnewline
$\CC(\R^n)$ & \mbox{nonempty closed convex subsets} & $\partial f$ & \mbox{convex subdifferential} \tabularnewline
$\S$ & \mbox{isometric leaf of $u$} & $f^*$ & \mbox{convex conjugate} \tabularnewline
$\S(x)$ & \mbox{the Borel leaf map (\textsection\,\ref{sec:geometry})} & $\Gamma_f$ & \mbox{graph map $x\mapsto(x,f(x))$} \tabularnewline
$N(u)$ & \mbox{nondifferentiability set of $u$} & $G_p$ & \mbox{graph face $\Gamma_f(F_p)$} \tabularnewline
$\T^k$ & \mbox{family of $k$-dimensional leaves} & $P_C$ & \mbox{metric projection onto $C$} \tabularnewline
$T_k$ & \mbox{union of $k$-dimensional leaves} & $N_C$ & \mbox{normal cone to $C$} \tabularnewline
$V_{\S}$ & \mbox{tangent space of $\S$} & $\rho$ & \mbox{ambient potential, $h=e^{-\rho}$} \tabularnewline
$T_{\S}$ & \mbox{linear isometry induced by $u|_{\S}$} & $\Ric_{\mu,N}$ & \mbox{weighted Ricci tensor \eqref{eq:6.2}} \tabularnewline
$P_{\S}$ & \mbox{orthogonal projection onto $V_{\S}$} & $\CD(\kappa,N)$ & \mbox{curvature-dimension condition} \tabularnewline
$Q_{\S}$ & \mbox{projection onto $T_{\S}(V_{\S})$} & $p_{\S}$ & \mbox{affine factor $g_{\S}/m_{\S}$ \eqref{eq:6.5}} \tabularnewline
$c_u$ & \mbox{equality defect \eqref{eq:2.3}} & $j(a,z)$ & \mbox{chart geometric density (Cor.~\ref{cor:7.7})} \tabularnewline
$\nu$ & \mbox{leaf quotient $\S_\#\mu$} & $w$ & \mbox{smooth positive reference density} \tabularnewline
$\mu_{\S}$ & \mbox{leaf conditional probability} & $\sigma$ & \mbox{reference probability $w\lambda$} \tabularnewline
\end{longtable}
\endgroup

\section*{Declarations}
\textbf{Competing interests.}
The author has no relevant financial or non-financial interests
to disclose.

\textbf{Use of AI tools.}
OpenAI's ChatGPT was used during preparation for drafting and editing, exploration of arguments, checking calculations and references, and literature searches. The paper has been thoroughly checked by the author. Responsibility for the mathematical content and attribution rests with the author.

\end{document}